\documentclass[11pt]{amsart}
\usepackage{amsmath,amsfonts,amssymb,amsthm,amscd}
\usepackage[alphabetic]{amsrefs}
\usepackage{hyperref}
\usepackage{graphicx,color}
\usepackage{pictexwd,dcpic,epsf}

\newtheorem{theorem}{Theorem}[section]
\newtheorem{lemma}[theorem]{Lemma}
\newtheorem{cor}[theorem]{Corollary}
\newtheorem{prop}[theorem]{Proposition}
\newtheorem*{maint}{Main Theorem}
\newtheorem*{mainc}{Corollary}

\theoremstyle{definition}
\newtheorem{definition}[theorem]{Definition}
\newtheorem{remark}[theorem]{Remark}

\newcommand {\C}{\mathcal{C}}
\newcommand {\G}{\mathcal{G}}

\newcommand {\mr}{\mathrm}

\newcommand{\Sal}{\operatorname{Sal}}

\begin{document}

\title{
%	Helly, Garside, and Artin groups 
%	(or) 
	Helly meets Garside and Artin
}

\author{Jingyin Huang}
\address{Department of Mathematics, The Ohio State University, 100 Math Tower, 231 W 18th Ave, Columbus, OH 43210, U.S.}
\email{huang.929@osu.edu}

\author{Damian Osajda}
\address{Instytut Matematyczny,
	Uniwersytet Wroc\l awski\\
	pl.\ Grun\-wal\-dzki 2/4,
	50--384 Wroc{\l}aw, Poland}
\address{Institute of Mathematics, Polish Academy of Sciences\\
	\'Sniadeckich 8, 00-656 War\-sza\-wa, Poland}
%\address{Dept.\ of Math.\ \& Stats., McGill University\\ Montreal, Quebec, Canada H3A 0B9}
\email{dosaj@math.uni.wroc.pl}
%2\date{\today}

\begin{abstract}
A graph is Helly if every family of pairwise intersecting combinatorial balls has a nonempty intersection.
We show that weak Garside groups of finite type and FC-type Artin groups are Helly, that is, they act geometrically on Helly graphs. In particular, such groups act geometrically on spaces with convex geodesic bicombing, equipping them with a nonpositive-curvature-like structure. 
That structure has many properties of a CAT(0) structure and, additionally, it has a combinatorial flavor implying biautomaticity.
As immediate consequences we obtain new results
for FC-type Artin groups (in particular braid groups and spherical Artin groups) and weak Garside groups, including e.g.\ fundamental groups of the complements of complexified finite simplicial arrangements of
hyperplanes, braid groups of well-generated complex reflection groups, and one-relator groups with non-trivial center. Among the results are:
biautomaticity, existence of EZ and Tits boundaries, the Farrell-Jones conjecture, the coarse Baum-Connes conjecture,
and a description of higher order homological and homotopical Dehn functions. 
As a means of proving the Helly property we introduce and use the notion of a (generalized) cell Helly complex.  
\end{abstract}

\maketitle
\tableofcontents
\setcounter{tocdepth}{2}

\section{Introduction}
\label{s:intro}

\subsection{Main result}
\label{s:main}
In this article we design geometric models for Artin groups and Garside groups, which are two generalizations of braid groups. 

Among the vast literature on these two classes of groups, we are particularly inspired by the results concerning the Garside structure, rooted in \cite{Garsidebraidgroup,Deligne,brieskorn1972artin}, and the small cancellation structure for certain Artin groups, exhibited in \cite{AppelSchupp1983,Appel1984,Pride1986}. These two aspects actually have interesting interactions, yielding connections with injective metric spaces and Helly graphs.

For geodesic metric spaces, being \emph{injective} (or, in other words, \emph{hyperconvex} or \emph{having the  binary intersection property}) can be characterized as follows: any collection of pairwise intersecting closed metric balls has a nonempty intersection. The graph-theoretic analogues of injective metric spaces are called \emph{Helly graphs} (Definition~\ref{d:Helly}). Groups acting on injective metric spaces were studied in \cite{Lang2013,descombes2016flats}, leading to a number of results parallel to the CAT(0) setting \cite{DesLang2015,MR3563261,miesch2018cartan}. The article \cite{Helly} initiates the studies of \emph{Helly groups}, that is, groups acting properly and cocompactly by graph automorphisms (i.e., \emph{geometrically}) on Helly graphs.

%For a graph $\Gamma$, by $V\Gamma$ we denote the set of its vertices. 

%
%A connected graph is \emph{Helly} if the family of its balls has the \emph{Helly property}: any collection of pairwise intersecting balls has a nonempty intersection. 
%Following \cite{Helly} --- where the studies of such groups were initiated --- by a \emph{Helly group}
%we mean a group acting properly and cocompactly by graph automorphisms
%(that is, \emph{geometrically}) on a Helly graph.
%A group acting properly and cocompactly by graph automorphisms
%(that is, \emph{geometrically}) on a Helly graph is called a \emph{Helly group}.

\begin{maint}
	Weak Garside groups of finite type and Artin groups of type FC are Helly.
\end{maint}

The Main Theorem should be seen as providing geometric models for groups in question with nonpositive-curvature-like structure. Consequences of this theorem are explained in Section~\ref{s:conseq}.

See Section~\ref{subsec:Garside basic} and Section~\ref{subsec:Artin def} for the definitions of weak Garside groups and Artin groups. Examples of (weak) Garside groups of finite type include:
\begin{enumerate}
	\item fundamental groups of the complements of complexified finite simplicial arrangements of
	hyperplanes \cite{Deligne};
	\item spherical Artin groups \cite{Deligne,brieskorn1972artin};
	\item braid groups of well-generated complex reflection groups \cite{bessis2015finite};
	\item structure groups
	of non-degenerate, involutive and braided set-theore\-ti\-cal solutions of
	the quantum Yang-Baxter equation \cite{MR2764830};
	\item one-relator groups with non-trivial center and tree products of cyclic groups \cite{picantin2013tree}.
\end{enumerate}

It is known that Garside groups are biautomatic \cite{Charney1992,Deho2002}. This suggests the non-positively curved (NPC) geometry of these groups. However, it is natural to ask for other notions of NPC for Garside groups in order to understand problems of coarse or continuous nature and to provide more convexity than the Garside normal form does. 

A few classes of Artin groups are known to be CAT(0) \cite{ChaDav1995,BraMcC2000,Brady2000,BraCri2002,Bell2005,BraMcC2010,HKS2016,haettel2019extra}. This leads to the speculation that all Artin groups should be NPC. However, due to the difficulty of verifying the CAT(0) condition in higher dimensional situations, it is not even known whether all braid groups are CAT(0), though some evidences are presented in \cite{dougherty2018boundary}. In particular, it is widely open whether Artin groups of type FC are CAT(0).

%It is widely believed that all Artin groups are in a sense nonpositively curved. One specific question is whether they are CAT(0), that is, they act geometrically on CAT(0) spaces. That was answered positively only for few particular subclasses of Artin groups: right-angled Artin groups (RAAGs) \cite{ChaDav1995}; certain classes of $2$-dimensional Artin groups \cite{BraCri2002,BraMcC2000}; Artin groups of finite type with three generators \cite{Brady2000}; $3$-dimensional Artin groups of type FC \cite{Bell2005}; spherical Artin groups of type $A_4$ and $B_4$ \cite{BraMcC2010}; $6$-strand braid group \cite{HKS2016}.

The Main Theorem applies to a large class of Artin groups and all finite type weak Garside groups (in particular all braid groups). Though the formulation is combinatorial, it has a continuous counterpart which is very similar to CAT(0). All Helly groups act geometrically on injective metric spaces \cite{Helly}, which in particular have convex geodesic bicombings. The convexity one obtains here is much stronger than biautomaticity and is weaker than CAT(0). See Subsection~\ref{s:conseq} for details.

On the other hand the Main Theorem equips the groups with a robust combinatorial structure --- the one of Helly
graphs. By a result from \cite{Helly} it implies e.g.\ biautomaticity --- an important algorithmic property
not implied by CAT(0) (see \cite{LearyMinasyan}). While allowing to deduce `CAT(0)-like' behavior as in the previous paragraph, the property of being Helly can be checked by local combinatorial conditions via a local-to-global characterization from \cite{weaklymodular}, see Theorem~\ref{thm:Helly} below.

%which are more user-friendly compared to the CAT(0) condition (that are difficult to check except in special situations). These combinatorial conditions are obtained 

The Main Theorem also provides a higher dimensional generalization of an old result of Pride \cite{Pride1986} saying that $2$-dimensional Artin groups of type FC are C(4)-T(4). The notion of cell Helly introduced in Section~\ref{s:Helly} serves as a higher dimensional counterpart of the C(4)-T(4) small cancellation condition in the sense of Lemma~\ref{lem:C4T4}. We prove that the universal covers of the Salvetti complexes of higher dimensional Artin groups of type FC are cell Helly.

Let us note that $2$-dimensional Artin groups of type FC are known to satisfy several different versions of NPC: small cancellation \cite{Pride1986}, CAT(0) \cite{BraMcC2000}, systolic \cite{HuaOsa1,Artinmetric}, quadric \cite{Hoda-q}, and Helly.

\subsection{Consequences of the Main Theorem}
\label{s:conseq}
In the Corollary below we list immediate consequences of being Helly, for groups as in the Main Theorem.
In \cite{Helly} it is shown that the class of Helly groups is closed under operations of
free products with amalgamations over finite subgroups, graph products (in particular, direct products), quotients by finite normal subgroups, and (as an immediate consequence of the definition) taking finite
index subgroups. Moreover, it is shown there that e.g.\ (Gromov) hyperbolic groups, CAT(0) cubical groups,
uniform lattices in buildings of type $\widetilde{C}_n$, 
and graphical C(4)-T(4) small cancellation groups are all Helly. Therefore, the results listed below in the Corollary
apply to various combinations of all those classes of groups.

Except item (6) below (biautomaticity) all other results (1)--(5) are rather direct consequences 
of well known facts concerning injective metric spaces --- see explanations in 
Subsection~\ref{s:det_cor} below, and \cite{Helly} for more details. Biautomaticity of Helly groups
is one of the main results of the paper \cite{Helly}.

\begin{mainc}
Suppose that a group $G$ is either a weak Garside group of finite type or an Artin group of type FC. Then the following hold.
	\begin{enumerate}
		\item $G$ acts geometrically on an injective metric space $X_G$, and hence on a metric space
		with a convex geodesic bicombing.
		Moreover, $X_G$ is a piecewise $\ell^\infty$-polyhedron complex.
		\item $G$ admits an EZ-boundary and a Tits boundary.
		\item The Farrell-Jones conjecture with finite wreath products holds for $G$.
		\item The coarse Baum-Connes conjecture holds for $G$.
		\item The $k$-th order homological Dehn function and homotopical Dehn function of $G$ are Euclidean, i.e.\ $f(l)\preceq l^{\frac{k+1}{k}}$.
		\item $G$ is biautomatic.
	\end{enumerate}
\end{mainc}

The recent work of Kleiner and Lang \cite{kleiner2018higher} provides many new insights into the asymptotic geometry of injective metric spaces. Together with (1) of the above corollary, it opens up the possibility of studying Gromov's program of quasi-isometric classification and rigidity for higher dimensional Artin groups which are not right-angled. This actually serves as one of our motivations to look for good geometric models for such groups. Previous works in this direction concern mainly the right-angled or the $2$-dimensional case \cite{behrstock2008quasi,MR2727658,MR3692971,huang2016quasi,huang2017quasi,margolis2018quasi}. 

We also present an alternative proof of the fact that the Salvetti complex for Artin groups of type FC is aspherical (see Theorem~\ref{t:cHFC}), which implies the $K(\pi,1)$-conjecture for such groups. This has been  previously established by Charney and Davis \cite{CharneyDavis} by showing contractibility of the associated Deligne complex. Our proof does not rely on the Deligne complex.  

\subsection{Details and comments on items in Corollary}
\label{s:det_cor}
\hfill

\medskip

\noindent
(1) \emph{Injective metric spaces} also known as \emph{hyperconvex metric spaces} or \emph{absolute retracts}
(in the category of metric spaces with $1$-Lipschitz maps) form an important class of metric spaces
rediscovered several times in history \cites{AroPan1956,Isbell1964,Dress1984} and studied thoroughly.
One important feature of injective metric spaces of finite dimension is that they admit a unique convex, consistent, reversible geodesic bicombing {\cite[Theorems 1.1\&1.2]{DesLang2015}}.
Recall that a \emph{geodesic bicombing} on a metric space $(X,d)$ is a map
\begin{align*}
\sigma \colon X\times X \times [ 0,1] \to X,
\end{align*}
such that for every pair $(x,y)\in X\times X$ the function $\sigma_{xy}:=\sigma(x,y,\cdot)$ is
a constant speed geodesic from $x$ to $y$. We call $\sigma$ \emph{convex} if the function 
$t\mapsto d(\sigma_{xy}(t),\sigma_{x'y'}(t))$ is convex for all $x,y,x',y'\in X$. The bicombing
$\sigma$ is \emph{consistent} if $\sigma_{pq}(\lambda)=\sigma_{xy}((1-\lambda)s+\lambda t)$, for all
$x, y \in X$, $0\leq s \leq t \leq 1$, $p := \sigma_{xy}(s)$, $q := \sigma_{xy}(t)$, and $\lambda \in [ 0, 1]$.
It is called \emph{reversible} if $\sigma_{xy}(t)=\sigma_{yx}(1-t)$ for all $x,y\in X$ and $t\in [0,1]$.

It is shown in \cite{Helly} that if a group $G$ acts geometrically on a Helly graph $\Gamma$ then it
acts geometrically on an \emph{injective hull} of $V(\Gamma)$ (set of vertices of $\Gamma$ equipped with a path metric) --- the smallest injective metric space containing isometrically $V(\Gamma)$. 
The proof is a straightforward consequence of the fact that the injective hull can be defined 
as a space of real-valued functions, the so-called \emph{metric forms} \cite{Isbell1964,Dress1984,Lang2013},
and similarly the smallest Helly graph containing isometrically a given graph --- the \emph{Helly hull} ---
can be defined analogously using integer metric forms (see e.g.\ \cite[Section II.3]{JPM}). 
It follows that $G$ acts on a space with convex geodesic bicombing.
Another example of a space with convex geodesic bicombing is a CAT(0) space with a combing consisting of
all (unit speed) geodesics. As noted in Subsection~\ref{s:main} only in very particular cases
it is possible to construct a CAT(0) space acted upon geometrically by an Artin group. We establish the
existence of a slightly weaker structure in much greater generality.

Holt \cite{Holt2010} proves falsification by fellow traveller for Garside groups. It follows then from 
\cite{Lang2013} that such groups act properly (not necessarily cocompactly) on injective metric spaces.
Garside groups have Garside normal forms, which can be thought as a combinatorial combing (not in the sense above). However, even coarsely this combing is not convex --- see some further discussion in (4) below.

The properties considered in (2)--(5) below are immediate consequences of acting geometrically on a space
with convex geodesic bicombing.

\medskip

\noindent
(2) Similarly to a CAT(0) space, an injective metric space (or, more generally, a space with convex geodesic bicombing) has a boundary at infinity, with two different topologies: the cone topology and the Tits topology \cite{DesLang2015,kleiner2018higher}. The former gives rise to an \emph{EZ-boundary} (cf. \cite{Bestvina1999,FaLa2005}) of $G$ \cite{DesLang2015}. The topology of an EZ-boundary reflects some algebraic properties of the group, and its existence implies e.g.\ the Novikov conjecture. The Tits topology is finer and provides subtle quasi-isometric invariants of $G$ \cite{kleiner2018higher}. It is observed in \cite{Helly} that Helly groups admit EZ-boundaries.

%The notion of an \emph{EZ-boundary} for a group is a natural extension of the Gromov boundary for (Gromov) hyperbolic groups and the visual boundary for CAT(0) groups. It emerged from works of Bestvina \cite{Bestvina1999} and Farrell-Lafont \cite{FaLa2005}. The topology of an EZ-boundary reflects some algebraic properties of the group, and the existence of such boundary implies e.g.\ the Novikov conjecture. Although it is conjectured that such boundaries exist for many groups, their existence is verified only for few classes. It is observed in \cite{Helly} that Helly groups admit EZ-boundaries.

\medskip

\noindent
(3) For a discrete group $G$ the \emph{Farrell-Jones conjecture} asserts that the $K$-theoretic (or $L$-theoretic) \emph{assembly map}, that is a homomorphism from a homology group of the classifying  space   
$E_{\mathcal{VCY}}(G)$ to a $K$-group (or $L$-group) of a group ring, is an isomorphism (see e.g.\ \cite{FarRou2000,BarLuc2012,KasRup2017,Roushon2018} for details). We say that $G$ satisfies the \emph{Farrell-Jones	conjecture with finite wreath products} if for any finite group $F$ the wreath product $G\wr F$ satisfies the Farrell-Jones conjecture.
%For a discrete group $G$ the \emph{Farrell-Jones conjecture} asserts that the $K$-theoretic (resp.\ $L$-theoretic) \emph{assembly map}
%\begin{align*}
%\label{e:FJK}
%H_n^G(E_{\mathcal{VCY}}(G);{\bf K}_R) \to K_n(RG)\\
%(\mathrm{resp.}\; H_n^G(E_{\mathcal{VCY}}(G);{\bf L}_R^{\langle -\infty \rangle}) \to L_n^{\langle -\infty \rangle}(RG))
%\end{align*}
%is an isomorphism. Here, $R$ is an associative ring with a unit, $RG$ is the group ring, and $K_n(RG)$ are the algebraic $K$--groups of $RG$. By $E_{\mathcal{VCY}}(G)$ we denote the classifying space for the family
%of virtually cyclic subgroups of $G$, and ${\bf K}_R$ is the spectrum given by algebraic $K$--theory with coefficients from $R$ (resp.\ we have the $L$-theoretic analogues) (see e.g.\ \cite{FarRou2000,BarLuc2012,KasRup2017,Roushon2018} for more details). We say that $G$ satisfies the \emph{Farrell-Jones	conjecture with finite wreath products} if for any finite group $F$ the wreath product $G\wr F$ satisfies the Farrell-Jones conjecture.
It is proved in \cite{KasRup2017} that groups acting on spaces with convex geodesic bicombing satisfy
the Farrell-Jones conjecture with finite wreath products. 

The Farrell-Jones conjecture implies, for example, the Novikov Conjecture, and the Borel Conjecture in high dimensions. 
Farrell-Roushon \cite{FarRou2000} established the Farrell-Jones conjecture for braid groups, Bartels-L{\"u}ck \cite{BarLuc2012} and Wegner \cite{Wegner2012} proved it for CAT(0) groups (see Subsection~\ref{s:main} for examples of CAT(0) Artin groups). Recently, Roushon~\cite{Roushon2018} established the Farrell-Jones conjecture
for some (but not all) spherical and Euclidean Artin groups. His proof relies on observing that some Artin groups fit into appropriate
exact sequences. After completion of the first version of the current paper Br{\"u}ck-Kielak-Wu \cite{XWu} 
presented another proof of the Farrell-Jones conjecture with finite wreath products for a subclass of the class of FC type Artin groups --- for even Artin groups of type FC.
\medskip

\noindent
(4) For a metric space $X$ the \emph{coarse assembly map} is a homomorphism from the coarse $K$-homology of $X$ to the $K$-theory of the Roe-algebra of $X$. The space $X$ satisfies the \emph{coarse Baum-Connes conjecture} if the coarse assembly map is an isomorphism. A finitely generated group $\Gamma$ satisfies the {coarse Baum-Connes conjecture} if the conjecture holds for $\Gamma$ seen as a metric space with the word metric given by a finite
generating set. Equivalently, the conjecture holds for $\Gamma$ if a metric space (equivalently: every metric space) acted geometrically upon by $\Gamma$ satisfies the conjecture. As observed in \cite{Helly}, it follows from the work of
Fukaya-Oguni \cite{FukayaOguni2019} that Helly groups satisfy the coarse Baum-Connes conjecture.

For CAT(0) groups the coarse Baum-Connes conjecture holds by \cite{HigRoe1995}. Groups with finite asymptotic dimension satisfy the coarse Baum-Connes conjecture. It is shown in \cite{BelFuj2008} that braid groups as well as spherical Artin groups $A_n,C_n$, and Artin groups of affine type $\widetilde{A}_n,\widetilde{C}_n$
have finite asymptotic dimension.

For proving the coarse Baum-Connes conjecture in \cite{FukayaOguni2019} a coarse version of a convex bicombing is used. Essential there is that two geodesics in the combing with common endpoints
`converge' at least linearly. This is not true for the usual combing in Garside groups. For example, such a combing for $\mathbb{Z}^2$ between points
$(n,n)$ and $(0,0)$ follows roughly the diagonal. Similarly the combing line between $(n+k,n)$ and
$(0,0)$ follows roughly a diagonal for $n\gg k$. Hence, for large $n$, the two lines stay at
roughly constant distance for a long time --- this behavior is different from what is required.   

\medskip

\noindent
(5) Higher order homological (resp.\ homotopical) Dehn functions measure the difficulty of filling higher dimensional cycles (resp.\ spheres) by chains (resp.\ balls), see \cite[Section 2]{abrams2013pushing,behrstock2017combinatorial} for backgrounds. The homological case of (5) is a consequence of (1) and \cite{wenger2005isoperimetric}, see the discussion in \cite[Section 2]{abrams2013pushing} for more details. The homotopical case of (5) follows from (1) and \cite{behrstock2017combinatorial}. The Garside case of (5) has been known before as it can be deduced from the biautomaticity and \cite{wenger2005isoperimetric,behrstock2017combinatorial}. The Artin case of (5) is new.

%\medskip
%\noindent
%(6) and (7) These two items concern quasi-isometric invariants of Garside groups and Artin groups, motivated by Gromov's program of classifying finitely generated groups up to quasi-isometries. A \emph{quasiflat} is a quasi-isometrically embedded Euclidean space. The asymptotic rank of $G$ equals the maximal possible dimension of quasiflats in $G$, see \cite{MR3563261} for several alternative characterizations. For NPC spaces/groups, the structure of top rank quasiflats usually play a fundamental role in understanding their asymptotic geometry, see \cite{kleiner2018higher} for background and literature.
%
%One can define Tits boundaries of injective metric spaces in a similar way to the CAT(0) case (cf.\ \cite[Section 8]{kleiner2018higher}). The term \emph{asymptotically conical} in (6) means that there exists a bi-Lipschitzly embedded sphere $S$ in the Tits boundary $\partial_T X_G$ such that for any base point $p\in X_G$, the geodesic cone $\mathrm{C}_p(S)$ over $S$ satisfies
%\begin{equation*}
%\lim_{R\to\infty}\frac{d_H(\bar B_p(R)\cap Q,\bar B_p(R)\cap \mathrm{C}_p(S))}{R}=0
%\end{equation*}
%Here $d_H$ is the Hausdorff distance and $\bar B_p(R)$ is the close ball centered at $p$ with radius $R$. (6) follows from \cite[Theorem 1.5]{kleiner2018higher} and (7) follows from \cite[Theorem 1.8]{kleiner2018higher}. The analogues of (6) and (7) in the special case of right-angled Artin groups were established in \cite{bks,MR3654109}.

\medskip

\noindent
(6) Biautomaticity is an important algorithmic property of a group. It implies, among others, that the Dehn
function is at most quadratic, and that the Word Problem and the Conjugacy Problem are solvable; see e.g.\ \cite{Epsteinetal1992}. It is proved in \cite{Helly} that all Helly groups are biautomatic.

Biautomaticity of all FC-type Artin groups is new (and so is the solution of the Conjugacy Problem and bounding the Dehn function for these Artin groups). 
Altobelli \cite{Alto1998} showed that FC-type Artin groups are asynchronously automatic, and hence have solvable Word Problem. 
%Asynchronous automaticity implies neither solvability of the Conjugacy Problem nor quadraticity of the Dehn function.
Biautomaticity for few classes of Artin groups was
shown before by: Pride together with Gersten-Short (triangle-free Artin groups) \cite{Pride1986,GesSho1990}, Charney \cite{Charney1992} (spherical), Peifer \cite{Peifer1996} (extra-large type), Brady-Mc\-Cam\-mond \cite{BraMcC2000} (three-generator large-type Artin groups and some generalizations), Huang-Osajda \cite{HuaOsa1} (almost large-type). 
Weak Garside groups of finite type are biautomatic by \cite{Charney1992,Deho2002}.

\subsection{Sketch of the proof of the main result}
Here we only discuss the special case of braid groups. Each braid group is the fundamental group of its Salvetti complex $X$. The universal cover $\widetilde X$ admits a natural cellulation by Coxeter cells. The key point is that the combinatorics of how these Coxeter cells in $\widetilde X$ intersect each other follows a special pattern which is a reminiscent of $CAT(0)$ cube complexes (but much more general). We formulate such pattern as the definition of cell Helly complex, see Definition~\ref{d:cellhelly} and Remark~\ref{rmk:no corner}. By the local-to-global characterization of Helly graphs established in \cite{weaklymodular}, the proof of the main theorem reduces to showing that $\widetilde X$ is cell Helly, and showing that the cell Hellyness can be further reduced to a calculation using Garside theory. 
\subsection{Organization of the paper}
In Section~\ref{s:preliminary}, we provide some background on Coxeter groups and Artin groups. In Section~\ref{s:Helly}, we discuss Helly graphs, introduce the notion of cell Helly complex, and present some natural examples. In Section~\ref{s:Garside}, we prove the Garside case of the Main Theorem and in Section~\ref{s:ArtinH} we deal with the Artin group case.

\subsection{Acknowledgments} 
J.H.\ thanks the Max-Planck Institute of Mathematics (Bonn), where part of this work was done, for its hospitality. J.H.\ also thanks R.\ Charney for helpful discussions on Garside groups. 
We would like to thank V.\ Chepoi, T.\ Fukaya, and U.\ Lang for comments on earlier
versions of the paper.
Part of the work on the paper was carried out while D.O.\ was visiting McGill University.
We would like to thank the Department of Mathematics and Statistics of McGill University
for its hospitality during that stay.
D.O.\ was partially supported by (Polish) Narodowe Centrum Nauki, UMO-2017/25/B/ST1/01335.
We thank the organizers of the conference ``Nonpositively Curved Groups in the Mediterranean", Nahsholim 2017,
for inviting us. A part of the work on the paper was carried out there.
This work was partially supported by the grant 346300 for IMPAN from the Simons Foundation and the matching 2015-2019 Polish MNiSW fund.

\section{Preliminaries}
\label{s:preliminary}
All graphs in this article are \emph{simplicial}, that is, they do not contain loops or multiple edges.
For a connected graph we equip the set of its vertices with a path metric: 
the distance of two vertices is the length (the number of edges) of the shortest edge-path between those vertices. A \emph{full subgraph} of a graph $\Gamma$ is a subgraph $\Gamma'$ such that two vertices are adjacent in $\Gamma'$ if and only if they are adjacent in $\Gamma$.

\subsection{Lattices}
A standard reference for the lattice theory is \cite{Birkhoff1967}.
Let $(P,\leq)$ be a partially ordered set (poset), and let $L$ be an arbitrary subset. An element $u\in P$ is an \emph{upper bound} of $L$ if $s\le u$ for each $s\in L$. An upper bound $u$ of $L$ is its \emph{least upper bound}, denoted by $\bigvee L$, if $u\le x$ for each upper bound $x$ of $L$. Similarly, $v\in P$ is a \emph{lower bound} of $L$ if $v\le s$ for each $s\in L$. A lower bound $v$ of $L$ is its \emph{greatest lower bound}, denoted by $\bigwedge L$, if $x \le v$ for each lower bound $x$ of $L$. Given $a,b\in P$, the \emph{interval} between $a$ and $b$, denoted $[a,b]$, is the collection of all elements $x\in P$ satisfying $a\le x\le b$.

A poset $(P,\leq)$ is a \emph{lattice} if each two-element subset $\{a,b\}\subset P$ has a least upper bound --- \emph{join}, denoted by $a\vee b$ --- and a greatest lower bound --- \emph{meet}, denoted by $a\wedge b$. 

\begin{lemma}[{\cite[Theorem IV.8.15]{Birkhoff1967}}]
	\label{lem:lattice}
	Suppose $(P,\leq)$ is a lattice. Then
	\begin{enumerate}
		\item if two intervals $[a,b]$ and $[c,d]$ of $P$ have nonempty intersection, then $[a,b]\cap [c,d]=[a\vee c,b\wedge d]$;
		\item if elements of a finite collection $\{[a_i,b_i]\}_{i=1}^n$ of intervals in $P$ pairwise intersect, then the intersection of all members in the collection is a non-empty interval $\bigcap_{i=1}^n [a_i,b_1]=[\bigvee \{a_i\}, \bigwedge \{b_i\}]$.
	\end{enumerate}
\end{lemma}

\subsection{Artin groups and Coxeter groups}
\label{subsec:Artin def}
Let $\Gamma$ be a finite simple graph with each of its edges labeled by an integer $\ge 2$. Let $V\Gamma$ be the vertex set of $\Gamma$. The \emph{Artin group with defining graph $\Gamma$}, denoted $A_{\Gamma}$, is given by the following presentation:
\begin{center}
	$\langle s_i\in V\Gamma\ |\ \underbrace{s_is_js_i\cdots}_{m_{ij}}=\underbrace{s_js_is_j\cdots}_{m_{ij}}$ for each $s_i$ and $s_j$ spanning an edge labeled by $m_{ij}\rangle$.
\end{center}
The \emph{Artin monoid with defining graph $\Gamma$}, denoted $A^+_\Gamma$ is the monoid with the same presentation. The natural map $A^+_\Gamma\to A_\Gamma$ is injective \cite{paris2002artin}.

The \emph{Coxeter group with defining graph $\Gamma$}, denoted $W_\Gamma$, is given by the following presentation:
\begin{center}
	$\langle s_i\in V\Gamma\ |\ s^2_i=1$ for any $s_i\in V\Gamma$, $ (s_is_j)^{m_{ij}}=1$ for each $s_i$ and $s_j$ spanning an edge labeled by $m_{ij}\rangle$.
\end{center}
Note that there is a quotient map $A_\Gamma\to W_\Gamma$ whose kernel is normally generated by the squares of generators of $A_\Gamma$.

\begin{theorem}
	\label{thm:lek}
	Let $\Gamma_1$ and $\Gamma_2$ be full subgraphs of $\Gamma$ with the induced edge labelings. Then
	\begin{enumerate}
		\item the natural homomorphisms $A_{\Gamma_1}\to A_\Gamma$ and $W_{\Gamma_1}\to W_\Gamma$ are injective;
		\item $A_{\Gamma_1}\cap A_{\Gamma_2}=A_{\Gamma_1\cap\Gamma_2}$ and $W_{\Gamma_1}\cap W_{\Gamma_2}=W_{\Gamma_1\cap\Gamma_2}$.
	\end{enumerate} 
\end{theorem}
The Coxeter case of this theorem is standard, see e.g. \cite{bourbaki2002lie} or \cite{dbook}, the Artin case of this theorem is proved in \cite{lek}.

A \emph{standard parabolic subgroup} of an Artin group or a Coxeter group is a subgroup generated by a subset of its generating set. 

\begin{theorem}\cite{MR3291260}
	\label{thm:convexity}
	Let $\Gamma'\subset \Gamma$ be a full subgraph. Let $X'\to X$ be the embedding from the Cayley graph of $A_{\Gamma'}$ (with respect to the standard presentation) to the Cayley graph of $A_\Gamma$ induced by the monomorphism $A_{\Gamma'}\to A_\Gamma$. Then $X'$ is convex in $X$, i.e.\ any geodesic of $X$ between two vertices of $X'$ is contained in $X'$.
\end{theorem}

Given $x,y\in A_\Gamma$, we define $x\preccurlyeq y$ (resp.\ $y\succcurlyeq x$) if $y=xz$ (resp.\ $y=zx$) for some $z$ in the Artin monoid $A^+_\Gamma$. By considering the length homomorphism $A_\Gamma\to \mathbb Z$ which sends each generator of $A_\Gamma$ to $1$, one deduces that $\preccurlyeq$ and $\succcurlyeq$ are both partial orders on $A_\Gamma$. They are called the \emph{prefix order} and the \emph{suffix order}, respectively. The prefix order is invariant under left action of $A_\Gamma$ on itself.

Now we recall the \emph{weak order} on a Coxeter group $W_\Gamma$. Each element $w\in W_\Gamma$ can be expressed as a word in the free monoid on the generating set of $W_\Gamma$. One shortest such expression (with respect to the word norm) is called a \emph{reduced expression} of $w$. Given $x,y\in W_\Gamma$, define $x\preccurlyeq y$ (resp.\ $y\succcurlyeq x$), if some reduced expression of $y$ has its prefix (resp.\ suffix) being a reduced expression of $x$. Note that $\preccurlyeq$ and $\succcurlyeq$ give two partial orders on $W_\Gamma$, called the \emph{right weak order} and the \emph{left weak order}, respectively.

The defining graph $\Gamma$ is \emph{spherical} if $W_\Gamma$ is a finite group, in this case, we also say the Artin group $A_\Gamma$ is \emph{spherical}. An Artin group is of \emph{type FC} if any complete subgraph $\Gamma'\subset\Gamma$ is spherical.
\begin{theorem} 
	\label{thm:lattice}
	Suppose $\Gamma$ is spherical. Then
	\begin{enumerate}
		\item $(W_\Gamma,\preccurlyeq)$ and $(W_\Gamma,\succcurlyeq)$ are both lattices;
		\item $(A_\Gamma,\preccurlyeq)$ and $(A_\Gamma,\succcurlyeq)$ are both lattices.
	\end{enumerate}
\end{theorem}
The first assertion is standard, see e.g.\ \cite{bjorner2006combinatorics}. The second assertion follows from \cite{brieskorn1972artin,Deligne}. 

\subsection{Davis complexes and oriented Coxeter cells}
\label{subsec:complex}
By a cell, we always mean a closed cell unless otherwise specified.

\begin{definition}[Davis complex]
	Given a Coxeter group $W_\Gamma$, let $\mathcal{P}$ be the poset of left cosets of spherical standard parabolic subgroups in $W_\Gamma$ (with respect to inclusion) and let $\Delta_\Gamma$ be the geometric realization of this poset (i.e.\ $\Delta_\Gamma$ is a simplicial complex whose simplices correspond to chains in $\mathcal{P}$). Now we modify the cell structure on $\Delta_\Gamma$ to define a new complex $D_\Gamma$, called the \emph{Davis complex}. The cells in $D_\Gamma$ are full subcomplexes of $\Delta_\Gamma$ spanned by a given vertex $v$ and all other vertices which are $\le v$ (note that vertices of $\Delta_\Gamma$ correspond to elements in $\mathcal{P}$, hence inherit the partial order).
\end{definition}

Suppose $W_\Gamma$ is finite with $n$ generators. Then there is a canonical faithful orthogonal action of $W_\Gamma$ on the Euclidean space $\mathbb E^n$. Take a point in $\mathbb E^n$ with trivial stabilizer, then the convex hull of the orbit of this point under the $W_\Gamma$ action (with its natural cell structure) is isomorphic to $D_\Gamma$. In such case, we call $D_\Gamma$ a \emph{Coxeter cell}. 

In general, the 1-skeleton of $D_\Gamma$ is the unoriented Cayley graph of $W_\Gamma$ (i.e.\ we start with the usual Cayley graph and identify the double edges arising from $s^2_i$ as single edges), and $D_\Gamma$ can be constructed from the unoriented Cayley graph by filling Coxeter cells in a natural way. Each edge of $D_\Gamma$ is labeled by a generator of $W_\Gamma$. We endow $D^{(1)}_\Gamma$ with the path metric.

We can identify the unoriented Cayley graph of $W_\Gamma$ with the Hasse diagram of $W_\Gamma$ with respect to right weak order. Thus we can orient each edge of $D_\Gamma$ from its smaller vertex to its larger vertex. When $W_\Gamma$ is finite, $D_\Gamma$ with such edge orientation is called an \emph{oriented Coxeter cell}. Note that there is a unique \emph{source vertex} (resp.\ \emph{sink vertex}) of $D_\Gamma$ corresponding to the identity element (resp.\ longest element) of $W_\Gamma$, where each edge containing this vertex is oriented away from (resp.\ oriented towards) this vertex. The source and sink vertices are called \emph{tips} of $D_\Gamma$.

The following fact is standard, see e.g.\ \cite{bourbaki2002lie}.

\begin{theorem}
	\label{thm:Coxeter convexity}
	Let $\mathcal R$ be a left coset of a standard parabolic subgroup, and let $v\in W_\Gamma$ be an arbitrary element. Then the following hold.
	\begin{enumerate}
		\item $\mathcal R$ is convex in the sense that for two vertices $v_1,v_2\in\mathcal R$, the vertex set of any geodesic in $D^{(1)}_\Gamma$ joining $v_1$ and $v_2$ is inside $\mathcal R$.
		\item There is a unique vertex $u\in\mathcal R$ such that $d(v,u)=d(v,\mathcal R)$, where $d$ denotes the path metric on the 1-skeleton of $D_\Gamma$.
		\item Pick $v_1\in\mathcal R$. Then $d(v,v_1)=d(v,u)+d(u,v_1)$. 
	\end{enumerate}
\end{theorem}

Theorem~\ref{thm:Coxeter convexity} implies that the face of an oriented Coxeter cell $C$, with its edge labeling and orientation from $C$, can be identified with a lower dimensional Coxeter cell in a way which preserves edge labeling and orientation.

The following lemma is well-known. For example, it follows from \cite[Lemma 3.4.9 and Corollary 5.5.16]{buekenhout2013diagram}. From another perspective, Theorem~\ref{thm:Coxeter convexity} (2)\&(3) says that
left cosets are the so-called \emph{gated subgraphs} of $D^{(1)}_\Gamma$. It is a well-known fact 
(see e.g.\ \cite[Proposition I.5.12(2)]{vandeVel1993}) that any family of gated subgraphs has the finite Helly property. In fact, the proof provided below (for the convenience of the reader), is the proof of this feature
for families of gated subgraphs. 
\begin{lemma}
	\label{lem:Coxeter helly}
	Let $\{R_i\}_{i=1}^n$ be a finite collection of left cosets of standard parabolic subgroups of $W_\Gamma$. Then $\cap_{i=1}^n R_i\neq\emptyset$ given $\{R_i\}_{i=1}^n$ pairwise intersect.
\end{lemma}
\begin{proof}
	We first look at the case $n=3$. Pick $x\in R_2\cap R_3$. Let $y$ be the unique point in $R_1$ closest to $x$ (Theorem~\ref{thm:Coxeter convexity} (2)). For $i=2,3$, let $y_i$ be the unique point in $R_1\cap R_i$ closest to $x$. By Theorem~\ref{thm:Coxeter convexity} (3), $d(x,y_2)=d(x,y)+d(y,y_2)$. Thus $y\in R_2$ by the convexity of $R_2$ (Theorem~\ref{thm:Coxeter convexity} (1)). Thus $y=y_2$. Similarly, $y=y_3$. Hence $y\in \cap_{i=1}^3 R_i$. For $n>3$, we induct on $n$. For $2\le i\le n$, let $\tau_i=R_1\cap R_i$. Then each $\tau_i$ is a left coset of a standard subgroup by Theorem~\ref{thm:lek}. Moreover, $\{\tau_i\}_{i=2}^n$ pairwise intersect by the $n=3$ case. Thus $\cap_{i=2}^n\tau_i\neq\emptyset$ by induction, and the lemma follows.
\end{proof}

\subsection{Salvetti complexes}
\begin{definition}[Salvetti complex]
	Given an Artin group $S_\Gamma$, we define a cell complex, called the \emph{Salvetti complex} and denoted by $S_\Gamma$, as follows. The 1-skeleton of $S_\Gamma$ is a wedge sum of circles, one circle for each generator. We label each circle by its associated generators and choose an orientation for each circle. Suppose the $(n-1)$-skeleton of $S_\Gamma$ has been defined. Now for each spherical subgraph $\Gamma'\subset\Gamma$ with $n$ vertices, we attach an oriented Coxeter cell $D_{\Gamma'}$ with the attaching map being a cellular map whose restriction to each face of the boundary of $D_{\Gamma'}$ coincides with the attaching map of a lower dimensional oriented Coxeter cell (when $n=2$, we require the attaching map preserves the orientation and the labeling of edges).
\end{definition}

The 2-skeleton of $S_\Gamma$ is the presentation complex of $A_\Gamma$. Hence $\pi_1(S_\Gamma)=A_\Gamma$. Let $X_\Gamma$ be the universal cover of $S_\Gamma$. The 1-skeleton of $X_\Gamma$ is the Cayley graph of $A_\Gamma$, and it is endowed with the path metric with each edge having length $=1$. We pull back labeling and orientation of edges in $S_\Gamma$ to $X_\Gamma$. A \emph{positive path} in $X^{(1)}_\Gamma$ from a vertex $x\in X_\Gamma$ to another vertex $y$ is an edge path such that one reads of a word in the Artin monoid by following this path. By considering the length homomorphism $A_\Gamma\to \mathbb Z$, we know each positive path is geodesic. 

We identify $X^{(0)}_\Gamma$ with elements in $A_\Gamma$. Hence $X^{(0)}_\Gamma$ inherits the prefix order $\preccurlyeq$ and the suffix order $\succcurlyeq$. Then $x\preccurlyeq y$ if and only if there is a positive path in $X^{(1)}_\Gamma$ from $x$ to $y$. 

The quotient homomorphism $A_\Gamma\to W_\Gamma$ induces a graph morphism $X^{(1)}_\Gamma\to D^{(1)}_\Gamma$ which preserves labeling of edges. Take a cell of $S_\Gamma$ represented by its attaching map $q:D_{\Gamma'}\to S_\Gamma$ with $D_{\Gamma'}$ being an oriented Coxeter cell. Let $\tilde q:D_{\Gamma'}\to X_\Gamma$ be a lift of $q$. Note that $\tilde q$ respects orientation and labeling of edges. It follows from definition that $D^{(1)}_{\Gamma'}\to X^{(1)}_\Gamma\to D^{(1)}_\Gamma$ is an embedding, hence $D^{(1)}_{\Gamma'}\to X^{(1)}_\Gamma$ is an embedding. By induction on dimension, we know $\tilde q$ is an embedding. Hence each cell in $X_\Gamma$ is embedded. 
Given a vertex $x\in X_\Gamma$, there is a one to one correspondence between cells of $X_\Gamma$ whose source vertex is $x$ and spherical subgraphs of $\Gamma$. Then $X^{(1)}_\Gamma\to D^{(1)}_\Gamma$ extends naturally to a cellular map $\pi:X_\Gamma\to D_\Gamma$ which is a homeomorphism on each closed cell.

\begin{lemma}
	\label{lem:isometric}
	Let $q:C\to X_\Gamma$ be the attaching map from an oriented Coxeter cell $C$ to $X_\Gamma$. Let $a$ and $b$ be the source vertex and the sink vertex of $C$. Let $\preccurlyeq_C$ be the right weak order on $C^{(0)}$. Then the following hold (we use $d$ to denote the path metric on the $1$-skeleton).
	\begin{enumerate}
		\item the map $(C^{(1)},d)\to (X^{(1)}_\Gamma,d)$ is an isometric embedding;
		\item for any two vertices $x,y\in C$, $x\preccurlyeq_C y$ if and only if $q(x)\preccurlyeq q(y)$, moreover, $[q(x),q(y)]_{\preccurlyeq}=q([a,b]_{\preccurlyeq_C})$ where $[q(x),q(y)]_\preccurlyeq$ denotes the interval between two elements with respect to the order $\preccurlyeq$;
		 \item for two cells $C_1$ and $C_2$ in $X_\Gamma$, if $C^{(0)}_1\subset C^{(0)}_2$, then there exists a cell $C_3\subset C_2$ such that $C^{(0)}_3=C^{(0)}_1$.
	\end{enumerate}
\end{lemma}

\begin{proof}
	Assertion (1) follows from that the map $X^{(1)}_\Gamma\to D^{(1)}_\Gamma$ is 1-Lipschitz. Now we prove (2). The only if direction follows from that $q$ preserves orientation of edges. Now assume $q(x)\preccurlyeq q(y)$. Since $q(a)\preccurlyeq q(x)\preccurlyeq q(y)\preccurlyeq q(b)$, there is a positive path $\omega$ from $q(a)$ to $q(b)$ passing through $q(x)$ and $q(y)$. Let $\bar C^{(1)},\bar \omega,\bar a$ and $\bar b$ be the images of $q(C^{(1)}),\omega,q(a)$ and $q(b)$ under $X^{(1)}_\Gamma\to D^{(1)}_\Gamma$. Then $d(a,b)=d(\bar a,\bar b)$, which also equals to the length of $\omega$ (recall each positive path is geodesic). Thus $length(\omega)=length(\bar \omega)$. By Theorem~\ref{thm:Coxeter convexity}, $\bar{\omega}\subset\bar C^{(1)}$. Then $\bar{\omega}$ (hence $\omega$) corresponds to the longest word in the finite Coxeter group associated $C$. Thus $\omega\subset q(C)$. It follows that $x\preccurlyeq y$. The proof of the moreover statement in (2) is similar.
	
	For (3), there is a positive path $\omega$ from the source of $C_1$ to the sink of $C_1$ corresponding to the longest word of some finite Coxeter group. The proof of (2) implies that $\omega\subset C_2$. Then $\omega$ determines $C_3\subset C_2$ as required.
\end{proof}

\begin{lemma}
	\label{lem:same source}
Let $C_1$ and $C_2$ be two cells of $X_\Gamma$. Then the following are equivalent:
\begin{enumerate}
	\item $C_1=C_2$;
	\item $C^{(0)}_1=C^{(0)}_2$;
	\item $C_1$ and $C_2$ have the same sink.
\end{enumerate}
\end{lemma}

\begin{proof}
$(1)\Rightarrow(2)$ is trivial. $(3)\Rightarrow(1)$	follows from the definition of cells. Now we prove $(2)\Rightarrow(3)$. First assume $C_1$ and $C_2$ are two edges. By considering the label preserving map $X^{(1)}_\Gamma\to D^{(1)}_\Gamma$, we know $C_1$ and $C_2$ have the same label. On the other hand, Theorem~\ref{thm:lek} (1) implies that map from the circle associated with a generator in $S_\Gamma$ to $S_\Gamma$ is $\pi_1$-injective. Thus $C_1=C_2$. Now we look at the general case. By Lemma~\ref{lem:isometric} (3), two vertices are adjacent in $C_1$ if and only if they are adjacent in $C_2$. Then $(2)\Rightarrow(3)$ follows from the 1-dimensional case by considering the orientation of edges on the cells.
\end{proof}

A \emph{full subcomplex} $K\subset X_\Gamma$ is a subcomplex such that if a cell of $X_\Gamma$ has vertex set contained in $K$, then the cell is contained in $K$. It is clear that intersection of two full subcomplexes is also a full subcomplex. The following is a consequence of Lemma~\ref{lem:isometric} (2) and Lemma~\ref{lem:same source}.

\begin{lemma}
	\label{lem:full}
	Any cell in $X_\Gamma$ is a full subcomplex. 
\end{lemma}

%\begin{proof}
%	We first claim two cells in $X_\Gamma$ with the same vertex set must be identical. Suppose the contrary holds. Then we can find two cells $C_1$ and $C_2$ of smallest possible dimension which violate the claim. The dimensional assumption implies that $\partial C_1=\partial C_2$.
%	
%	%	Note that each oriented Coxeter cell has a unique source vertex. Thus for two cells $C_1,C_2\subset X_\Gamma$, if  (as subcomplexes of $X_\Gamma$), then $C_1=C_2$. 
%	%	
%	%Suppose the first statement of the lemma does hold. Then we can find a pair of cells $C_1$ and $C_2$ such that $C^{(0)}_1\subset C^{(0)}_2$, $C_1\nsubseteq C_2$. We assume in addition that $\dim(C_1)+\dim(C_2)$ attains its minimal among all such pairs.
%	%
%	%
%	%We argue by contradiction and suppose $C$ is not full. Let $C'$ be a cell of smallest possible dimensional such that the vertex set of $C'$ is in $C$ but $C'\nsubseteq C$. The dimension assumption on $C'$ implies that $\partial C'\subset C$. Let $a$ and $b$ 
%\end{proof}

%???Moreover, the graph morphism $X^{(1)}_\Gamma\to D^{(1)}_\Gamma$ extends to a cellular map $X_\Gamma\to D_\Gamma$ which sends each cell in $X_\Gamma$ homeomorphically to a cell in $D_\Gamma$. 

Take a full subgraph $\Gamma'\subset\Gamma$, then there is an embedding $S_{\Gamma'}\to S_\Gamma$ which is $\pi_1$-injective (Theorem~\ref{thm:lek}~(1)). A \emph{standard subcomplex of type $\Gamma'$} of $X_\Gamma$ is a connected component of the preimage of $S_{\Gamma'}$ under the covering map $X_\Gamma\to S_\Gamma$. Each such standard subcomplex can be naturally identified with the universal cover $X_{\Gamma'}$ of $S_{\Gamma'}$. A standard subcomplex is \emph{spherical} if its type is spherical. We define the \emph{type} of a cell in $X_\Gamma$ to be the type of the smallest standard subcomplex that contains this cell. 

\begin{lemma}
	\label{lem:standard subcomplex is full}
Any standard subcomplex of $X_\Gamma$ is a full subcomplex.
\end{lemma}

\begin{proof}
Let $X\subset X_\Gamma$ be a standard subcomplex of type $\Gamma'$. Let $C\subset X_\Gamma$ be a cell such that $C^{(0)}\subset X$. By considering the map $X^{(1)}_\Gamma\to D^{(1)}_\Gamma$, we know each edge of $C$ is labeled by a vertex of $\Gamma'$. Thus $C$ is of type $\Gamma''$ for $\Gamma''\subset\Gamma'$. Hence $C$ is mapped to $S_{\Gamma''}$ under $X_\Gamma\to S_\Gamma$. Since $C\cup X$ is connected, $C\subset X$. Then the lemma follows.
\end{proof}

\begin{lemma}
	\label{lem:intersection of cell and subcomplex}
Let $X_1, X_2\subset X_\Gamma$ be standard subcomplexes of type $\Gamma_1,\Gamma_2$. 
	\begin{enumerate}
		\item Suppose $X_1\cap X_2\neq \emptyset$. Then $X_1\cap X_2$ is a standard subcomplex of type $\Gamma_1\cap \Gamma_2$.
		\item Let $C\subset X_\Gamma$ be a cell of type $\Gamma_1$ such that $C\cap X_2\neq\emptyset$. Then $C\cap X_2$ is a cell of type $\Gamma_1\cap\Gamma_2$.
		\item For elements $a,b\in X^{(0)}_1$, the interval between $a$ and $b$ with respect to $(X^{(0)}_\Gamma,\preccurlyeq)$ is contained in $X^{(0)}_1$.
		\item We identify $X^{(0)}_1$ with $A_{\Gamma_1}$, hence $X^{(0)}_1$ inherits a prefix order $\preccurlyeq_{\Gamma_1}$ from $A_{\Gamma_1}$. Then $\preccurlyeq_{\Gamma_1}$ coincides with the restriction of $(A_\Gamma,\preccurlyeq)$ to $X^{(0)}_1$.
	\end{enumerate}	
\end{lemma}

\begin{proof}
By Theorem~\ref{thm:lek}, there exists a standard subcomplex $X_0\subset X_\Gamma$ such that $X^{(0)}_0=X^{(0)}_1\cap X^{(0)}_2$. Then (1) follows from Lemma~\ref{lem:standard subcomplex is full}. Now we prove (2). Let $X_1$ be the standard subcomplex of type $\Gamma_1$ containing $C$. Then $C\cap X_2$ is contained in $X_1\cap X_2$, which is a standard subcomplex of type $\Gamma_1\cap \Gamma_2$. This together with Lemma~\ref{lem:standard subcomplex is full} imply that $C\cap X_2$ is a disjoint union of faces of $C$ of type $\Gamma_1\cap \Gamma_2$. It remains to show $C\cap X_2$ is connected. Suppose the contrary holds. Pick a shortest geodesic edge path $\omega\subset C$ (with respect to the path metric on $C^{(1)}$) traveling from a vertex in one component of $C\cap X_2$ to another component. Then $\omega$ is also a geodesic path in $X^{(1)}_\Gamma$ by Lemma~\ref{lem:isometric}. Since two end points of $\omega$ are in $X_2$, we have $\omega\subset X_2$ by Theorem~\ref{thm:convexity}, which yields a contradiction. Assertion (3) and (4) follow from Theorem~\ref{thm:convexity}.
\end{proof}

\section{Helly graphs and cell-Helly complexes}
\label{s:Helly}
See the beginning of Section~\ref{s:preliminary} for our convention on graphs. A \emph{clique} of a graph is a complete subgraph. A \emph{maximal clique} is a clique not properly contained in another clique. Each graph is endowed with the path metric such that each edge has length 1. A \emph{ball} centered at $x$ with radius $r$ in a graph $\Gamma$ is the full subgraph spanned by all vertices at distance $\le r$ from $x$. As for Helly graphs we follow usually the notation from \cite{weaklymodular,Helly}. 

\begin{definition}
	\label{d:Helly}
	A family $\mathcal{F}$ of subsets of a set satisfies the (\emph{finite}) \emph{Helly property}
	if for any (finite) subfamily $\mathcal{F}'\subseteq \mathcal{F}$ of pairwise intersecting subsets the intersection $\bigcap \mathcal{F}'$ is nonempty. 
	
	A graph is (\emph{finitely}) \emph{clique Helly} if the family of all maximal cliques, as subsets of the vertex sets of the graph, satisfies the (finite) Helly property.
	
	A graph is (\emph{finitely}) \emph{Helly} if the family of all balls, as subsets of the vertex sets of the graph, satisfies the (finite) Helly property.
\end{definition}

The notion of clique Helly is local and the notion of Helly is global. Our main tool for proving Hellyness of some graphs is the following local-to-global characterization.

The \emph{clique complex} of a graph $\Gamma$ is the flag completion of $\Gamma$. The \emph{triangle complex} of $\Gamma$ is the $2$-skeleton of the clique complex, that is, the $2$-dimensional simplicial complex obtained by spanning $2$-simplices (that is, triangles) on all $3$-cycles (that is, $3$-cliques) of $\Gamma$.

\begin{theorem}
%	[{\cite[Theorem 3.5]{weaklymodular} and \cite[Proposition 3.1.2]{polat2001convexity}}]
	\label{thm:Helly}
	(i) (\cite[Theorem 3.5]{weaklymodular}) Let $\Gamma$ be a simplicial graph. Suppose that $\Gamma$ is finitely clique Helly, and that the clique  complex of $\Gamma$ is simply-connected. Then $\Gamma$ is finitely Helly. 
	
	(ii) (\cite[Proposition 3.1.2]{polat2001convexity})	If in addition $\Gamma$ does not have infinite cliques, then $\Gamma$ is Helly.
\end{theorem}

Parallel to this theorem, a local-to-global result for injective metric spaces is obtained in \cite{miesch2018cartan}.
\begin{lemma}
	\label{lem:helly criterion1}
	Let $X$ be a set and let $\{X_i\}_{i\in I}$ be a family of subsets of $X$. Define a simplicial graph $Y$ whose vertex set is $X$, and two points are adjacent if there exists $i$ such that they are contained in $X_i$. Suppose that
	\begin{enumerate}
		\item there are no infinite cliques in $Y$;
		\item $\{X_i\}_{i\in I}$ has finite Helly property, i.e.\ any finite family of  pairwise intersecting elements in $\{X_i\}_{i\in I}$ has nonempty intersection;
		\item for any triple $X_{i_1},X_{i_2},X_{i_3}$ of pairwise intersecting subsets in $\{X_i\}_{i\in I}$, there exists 
%		$X_{i_0}\in \{X_i\}_{i\in I}$ 
		$i_0 \in I$ such that 
%		$(X_{i_1}\cap X_{i_2})\cup(X_{i_2}\cap X_{i_3})\cup (X_{i_3}\cap X_{i_1})$ is contained in $X_{i_0}$.
		\begin{align*}
		(X_{i_1}\cap X_{i_2})\cup(X_{i_2}\cap X_{i_3})\cup (X_{i_3}\cap X_{i_1})\subseteq X_{i_0}.
		\end{align*}
	\end{enumerate}
	Then $Y$ is finitely clique Helly.
\end{lemma}

\begin{proof}
	Let $S\subset X$ be a subset. We claim if $S$ is the vertex set of a clique in $Y$, then there exists $i\in I$ such that $S\subset X_i$. We prove it by induction on $|S|$. Note that $|S|<\infty$ by assumption (1), and the case $|S|=2$ is trivial. Suppose $|S|=n\ge 3$. Let $\{x_1,x_2,x_3\}$ be three pairwise distinct vertices in $S$ and let $S_i=S\setminus\{x_i\}$. It follows from induction assumption that there exists $X_i$ such that $S_i\subset X_i$. Note that $\{X_{i_1},X_{i_2},X_{i_3}\}$ pairwise intersect. Thus by assumption (3), there exists $X_{i_0}\in \{X_i\}_{i\in I}$ such that $S\subset (X_{i_1}\cap X_{i_2})\cup(X_{i_2}\cap X_{i_3})\cup (X_{i_3}\cap X_{i_1})\subset X_{i_0}$. Hence the claim follows. The claim implies that if $S$ is the vertex set of a maximal clique in $Y$, then $S=X_i$ for some $i\in I$. Now the lemma follows from assumption (2).
\end{proof}

%\begin{theorem}
%	\label{thm:Helly}
%	Let $Y$ be a simplicial graph. Suppose that $Y$ is finitely clique Helly, and that the flag complex of $Y$ is simply-connected. Then $Y$ is finitely Helly. If in addition $Y$ does not have infinite cliques, then $Y$ is Helly.
%\end{theorem}
%The first statement follows from \cite[Theorem 3.5]{weaklymodular} and the second statement follows from \cite[Proposition 3.1.2]{polat2001convexity}.

The lemma above justifies the following definition that will be our main tool for
 showing that some groups act nicely on Helly graphs. A subcomplex of $X$ is \emph{finite} it has only finitely many cells.

\begin{definition}
	\label{d:cellHelly}
	Let $X$ be a combinatorial complex (cf.\ \cite[Chapter I.8.A]{BridsonHaefliger1999}) and let $\{X_i\}_{i\in I}$ be a family of finite full subcomplexes covering $X$. We call $(X,\{X_i\})$ (or simply $X$, when the family $\{X_i\}$ is evident) a \emph{generalized cell Helly} complex if the following conditions are satisfied:
	\begin{enumerate}
		%		\item there are no infinite cliques in $Y$;
		\item each intersection $X_{i_1}\cap X_{i_2}\cap \cdots \cap X_{i_k}$ is either empty or connected and simply connected (in particular, all members of $\{X_i\}$ are connected and simply connected);
		\item the family $\{X_i\}_{i\in I}$ has finite Helly property, i.e.\ any finite family of pairwise intersecting elements in $\{X_i\}_{i\in I}$ has nonempty intersection;
		\item for any triple $X_{i_1},X_{i_2},X_{i_3}$ of pairwise intersecting subcomplexes in $\{X_i\}_{i\in I}$, there exists 
%		$X_{i_0}\in \{X_i\}_{i\in I}$ 
		$i_0\in I$ such that 
%		$(X_{i_1}\cap X_{i_2})\cup(X_{i_2}\cap X_{i_3})\cup (X_{i_3}\cap X_{i_1})$ is contained in $X_{i_0}$.
		\begin{align*}
		(X_{i_1}\cap X_{i_2})\cup(X_{i_2}\cap X_{i_3})\cup (X_{i_3}\cap X_{i_1})\subseteq X_{i_0}.
		\end{align*}
	\end{enumerate}
	The subcomplexes $X_i$ are called \emph{generalized cells} of $X$.
\end{definition}

In this article we will only deal with the following special case of generalized cell Helly complex.

\begin{definition}
	\label{d:cellhelly}
	Suppose $X$ is a combinatorial complex such that each of its closed cell is embedded and each closed cell is a full subcomplex of $X$. Let $\{X_i\}_{i\in I}$ be the collection of all closed cells in $X$. We say $X$ is \emph{cell Helly} if $(X,\{X_i\})$ satisfies Definition~\ref{d:cellHelly}.
\end{definition}

A natural class of cell Helly complexes is described in the lemma below, whose proof is left to the reader. 
\begin{lemma}
Suppose $X$ is a simply-connected cube complex such that each cube in $X$ is embedded and any intersection of two cubes of $X$ is a sub-cube. Then $X$ is cell Helly if and only if $X$ is a $CAT(0)$ cube complex.
\end{lemma}

\begin{remark}
	\label{rmk:no corner}
	It can be seen here that the condition (3) from Definition~\ref{d:cellHelly} plays a role of Gromov's flagness condition. Intuitively, 
	it tells that all the corners (of possible positive curvature) are ``filled" (and hence curvature around them is nonpositive).
\end{remark}

Pride \cite{Pride1986} showed that $2$-dimensional Artin groups of type FC are C(4)-T(4) small cancellation groups. The following lemma together with Theorem~\ref{t:cHFC} can be seen as an extension of Pride's result to
higher dimensions. We refer to \cite{Pride1986,Hoda-q} for the definition of C(4)-T(4) small cancellation conditions. Here, we assume that cells are determined by their attaching maps (this corresponds to the usual small cancellation theory, for groups without torsion). In particular, by \cite[Proposition 3.5]{Hoda-q} it means that the intersection of two cells is either empty or a vertex or an
interval.

\begin{lemma}
	\label{lem:C4T4}
	Suppose $X$ is a $2$-dimensional simply-connected combinatorial complex such that each closed cell is embedded. Then $X$ is cell Helly if and only if $X$ satisfies the C(4)-T(4) small cancellation conditions.
\end{lemma}	

\begin{proof}
If $X$ is a C(4)-T(4) complex then, by \cite[Proposition 3.8]{Hoda-q}, for any pairwise intersecting cells
$C_0,C_1,C_2$ we have $C_i\cap C_j\subseteq C_k$, for $\{i,j,k\} = \{0,1,2\}$. It follows that $X$ is cell Helly.

Now suppose $X$ is cell Helly. 
Since nonempty intersections of cells are connected and simply connected, such intersections must be vertices or intervals.

We check the C(4) condition. 
First suppose there is a cell $C$ with the boundary covered by two intervals $C\cap C_0$ and $C\cap C_1$ being
intersections with cells $C_0,C_1$. Since, by Definition~\ref{d:cellHelly}(1), the intersection $C_0\cap C_1$
has to be a nontrivial interval. But then the union $(C\cap C_0)\cup (C_0\cap C_1)\cup (C_1\cap C)$ is a graph being a union of three intervals along their endpoints, contradicting Definition~\ref{d:cellHelly}(3).  
Now, suppose there are cells $C,C_0,C_1,C_2$ such that $C\cap C_i\neq \emptyset$ and $C_i\cap C_{i+1 (\mr{mod}\, 3)}\neq \emptyset$. It follows that there is a common intersection $C\cap C_0\cap C_1\cap C_2$. This brings us to the previous case. 

For the T(4) condition let suppose there are cells $C_0,C_1,C_2$ with $C_i\cap C_{i+1 (\mr{mod}\, 3)}$ being a nontrivial interval containing a vertex $v$, for all $i$. Then $(C_0\cap C_1)\cup (C_1\cap C_2)\cup (C_2\cap C_0)$ cannot be contained in a single cell, contradicting Definition~\ref{d:cellHelly}(3).
\end{proof}

Definition~\ref{d:cellHelly} generalizes Definition~\ref{d:cellhelly} in the same way as the graphical small cancellation theory generalizes the classical small cancellation theory, another example to consider is $X$ being a C(4)-T(4) graphical small cancellation complex. Defining $\{X_i\}$ as the family of simplicial cones over relators, we obtain a generalized cell Helly complex (proved in \cite{Helly}).

Later we will see more examples of cell Helly complexes whose cells are zonotopes (Corollary~\ref{cor:zonotopes}) or Coxeter cells (Theorem~\ref{t:cHFC}).

%Tell reader that we mainly use the case where $X_i$ to be maximal cells.
%
%Explain condition (3), invoking $CAT(0)$ cube complex.
%
%Graphical C(4)-T(4).
%
%Contractibility criterion. 

For a pair $(X,\{X_i\}_{i\in I})$ as above, we define its \emph{thickening} $Y$ to be the following simplicial complex. The $0$-skeleton of $Y$ is $X^{(0)}$, and two vertices of $Y^{(0)}$ are adjacent if they are contained in the same generalized cell of $X$. Then $Y$ is defined to be the flag complex of $Y^{(1)}$.
We call $(X,\{X_i\}_{i\in I})$ \emph{locally bounded} if there does not exist an infinite strictly increasing 
sequence of vertex sets $A_1\subsetneq A_2 \subsetneq A_3 \subsetneq \cdots$ such that for every 
$j$ there exists $i_j$ with $A_j\subseteq X_{i_j}$.

\begin{theorem}
	\label{t:cH>H}
	The thickening of a locally bounded generalized cell Helly complex is clique Helly. 
	The thickening of a simply connected locally bounded generalized cell Helly complex is Helly.
\end{theorem}
\begin{proof}
	Let $(X,\{X_i\}_{i\in I})$ be a locally bounded generalized cell Helly complex and let $Y$ be its thickening. It follows from the proof of Lemma~\ref{lem:helly criterion1} that the vertex set of each finite clique of $Y$ is contained in some $X_{i_j}$ (this does not use assumption (1) of Lemma~\ref{lem:helly criterion1}). This together with local boundedness imply that there are no infinite cliques in $Y$.
	The first statement follows from Lemma~\ref{lem:helly criterion1}.
	By the Nerve Theorem \cite[Theorem 6]{bjorner2003nerves} applied to the covering 
	of $X$ by generalized cells $\{X_i\}$, and the covering of $Y$ by simplices corresponding to generalized cells
	we get that $X$ and $Y$ are have the same connectivity and isomorphic fundamental groups.
	If $X$ is simply connected then so is $Y$ and hence, by Theorem~\ref{thm:Helly}, the thickening is Helly. 
\end{proof}

%
%Recall that a cell complex $X$ is \emph{regular} if each closed cell in $X$ is embedded. For a regular cell complex $X$, we define the \emph{thickening} $Y$ of $X$ to be a simplicial complex as follows. The $0$-skeleton of $Y$ is $X^{(0)}$, and two vertices of $Y^{(0)}$ are adjacent if they are contained in the same closed cell of $X$. Then $Y$ is defined to be the flag complex of $Y^{(1)}$.
%
%\begin{theorem}
%	Let $X$ be a simply connected locally finite regular cell complex. Let $Y$ be the thickening of $X$. Suppose that
%	\begin{enumerate}
%		\item the collection of maximal closed cells in $X$ has finite Helly property;
%		\item for any triple $C_1,C_2,C_3$ of pairwise intersecting maximal closed cell in $X$, there exists an maximal closed cell $C_0$ such that $(C_1\cap C_2)\cup(C_2\cap C_3)\cup (C_3\cap C_1)$ is contained in $C_0$.
%	\end{enumerate}
%	Then $Y$ is a Helly complex. If we assume in addition that any nonempty intersection of finitely many maximal closed cells in $X$ is contractible, then $X$ is homotopy equivalent to $Y$, hence $X$ is contractible.
%\end{theorem}
%\begin{proof}
%	The contractibility follows from Borsuk's Nerve Theorem \cite{Borsuk1948} applied to the covering 
%	of $X$ by maximal cells, and the covering of $Y$ by maximal simplices. 
%\end{proof}

We say that a group $G$ \emph{acts} on a generalized cell Helly complex $(X,\{X_i\}_{i\in I})$ if it acts by automorphisms
on $X$ and preserves the family $\{X_i\}_{i\in I}$. For such an action clearly $G$ acts by automorphisms
on the thickening of $(X,\{X_i\}_{i\in I})$. Moreover, we have the following.

\begin{prop}
	\label{p:Helly_action}
	If a group $G$ acts geometrically on a simply connected locally bounded generalized cell Helly complex $(X,\{X_i\}_{i\in I})$ then $G$ acts geometrically on the Helly graph being the $1$-skeleton $Y^{(1)}$ of the thickening $Y$. In particular, $G$ is Helly.
\end{prop}
\begin{proof}
	We first show that each vertex $v$ in $Y^{(1)}$ is adjacent to finitely many other vertices.  Suppose $v$ is adjacent to infinitely many different vertices $\{v_i\}_{i=1}^\infty$. Take $X_i$ which contains both $v$ and $v_i$. As $X$ is uniformly locally finite, we can assume $\lim_{i\to\infty}d_{X^{(1)}}(v,v_i)=\infty$ where $d_{X^{(1)}}$ denotes the path metric on $X^{(1)}$. Let $P_i$ be a shortest edge path in $X_i$ connecting $v$ and $v_i$. Then $P_i$ subconverges to an infinite path as $i\to\infty$, which contradicts the local boundedness assumption.
	
	Since the $G$-action on $X$ is cocompact, the quotient $Y^{(0)}/G$ of the vertex set is finite.
	Since $(X,\{X_i\})$ is locally bounded there are no infinite cliques (see the proof of Theorem~\ref{t:cH>H}) and hence the quotient 
	$Y^{(1)}/G$ is finite.
	
	To prove properness it remains to show that, for the $G$-action on $Y^{(1)}$, stabilizers of cliques
	are finite. Such stabilizers correspond to stabilizers of finite vertex sets for the $G$-action on $X$, which are finite by the properness of the $G$-action on $X$.
\end{proof}

We close the section with a simple but useful observation, that will allow us to reprove contractibility of the universal cover of the Salvetti complex of an FC-type Artin group in Section~\ref{s:ArtinH}.

\begin{prop}
	\label{p:contract}
	Let $(X,\{X_i\}_{i\in I})$ be a simply connected locally bounded generalized cell Helly complex such that the intersection
	of any collection of generalized cells is either empty or contractible. Then $X$ is contractible.
\end{prop}
\begin{proof}
	By Borsuk's Nerve Theorem \cite{Borsuk1948} applied to the covering 
	of $X$ by generalized cells $\{X_i\}$, and the covering of $Y$ by simplices corresponding to generalized cells, we have that $X$ is homotopically equivalent to its thickening. The latter is contractible as a Helly complex by Theorem~\ref{t:cH>H}.	
\end{proof}
%\begin{remark}
%	In the case f a regular cell complex $X$ we may consider the family of cells $\{X_i\}$ coinciding with
%	the closed (polyhedral) cells. This is the most typical situation that includes all the cases
%	of complexes treated in this article. In that situation cells and their intersections are contractible,
%	and any action of a group $G$ by cellular automorphisms on $X$ induces a $G$-action on the cell Helly complex.
%	It is easy to provide more complicated examples of complexes where our definition of call Helly
%	complexes could be useful.
%\end{remark}

\section{Weak Garside groups act on Helly graphs}
\label{s:Garside}
\subsection{Basic properties of Garside categories}
\label{subsec:Garside basic}
We follow the treatment of Garside categories in \cite{bessis2006garside,Krammer}. 

Let $\C$ be a small category. One may think of $\C$ as of an oriented graph, whose vertices are objects in $\C$ and oriented edges are morphisms of $\C$. Arrows in $\C$ compose like paths: $x\stackrel{f}{\to} y\stackrel{g}{\to} z$ is composed into $x\stackrel{fg}{\to} z$. For objects $x,y\in\C$, let $\C_{x\to}$ denote the collection of morphisms whose source object is $x$. Similarly we define $\C_{\to y}$ and $\C_{x\to y}$. 

For two morphisms $f$ and $g$, we define $f\preccurlyeq g$ if there exists a morphism $h$ such that $g=fh$. Define $g\succcurlyeq f$ if there exists a morphism $h$ such that $g=hf$. Then $(\C_{x\to},\preccurlyeq)$ and $(\C_{\to y},\succcurlyeq)$ are posets. A nontrivial morphism $f$ which cannot be factorized into two nontrivial factors is an \emph{atom}.
 
The category $\C$ is \emph{cancellative} if, whenever a relation $afb=agb$ holds between composed morphisms, it implies $f=g$. $\C$ is \emph{homogeneous} if there exists a length function $l$ from the set
 of $\C$-morphisms to $\mathbb Z_{\ge 0}$ such that $l(fg) = l(f) + l(g)$ and $(l(f) = 0)\Leftrightarrow$ ($f$ is a unit).

We consider the triple $(\C,\C\stackrel{\phi}{\to}\C,1_\C \stackrel{\Delta}{\Rightarrow}\phi)$ where $\phi$ is an automorphism of $\C$ and $\Delta$ is a natural transformation from the identity function to $\phi$. For an object $x\in \C$, $\Delta$ gives morphisms $x\stackrel{\Delta(x)}{\longrightarrow} \phi(x)$ and $\phi^{-1}(x)\stackrel{\Delta(\phi^{-1}(x))}{\longrightarrow} x$. We denote the first morphism by $\Delta_x$ and the second morphism by $\Delta^x$. A morphism $x\stackrel{f}{\to} y$ is \emph{simple} if there exists a morphism $y\stackrel{f^\ast}{\to} \phi(x)$ such that $f f^\ast=\Delta_x$. When $\C$ is cancellative, such $f^\ast$ is unique.

\begin{definition}
	\label{def:Garside}
A \emph{homogeneous categorical Garside structure} is a triple $(\C,\C\stackrel{\phi}{\to}\C,1_\C \stackrel{\Delta}{\Rightarrow}\phi)$ such that:
\begin{enumerate}
	\item $\phi$ is an automorphism of $\C$ and $\Delta$ is a natural transformation from the identity function to $\phi$;
	\item $\C$ is homogeneous and cancellative;
	\item all atoms of $\C$ are simple;
	\item for any object $x$, $\C_{x\to}$ and $\C_{\to x}$ are lattices.
\end{enumerate}
It has \emph{finite type} if the collection of simple morphisms of $\C$ is finite. 
\end{definition}

\begin{definition}
A \emph{Garside category} is a category	$\C$ that can be equipped with $\phi$ and $\Delta$ to obtain a homogeneous categorical Garside structure. A \emph{Garside groupoid} is the enveloping groupoid of a Garside category. See \cite[Section 3.1]{dehornoy2015foundations} for a detailed definition of enveloping groupoid. Informally speaking, it is a groupoid obtained by adding formal inverses to all morphisms in a Garside category.

Let $x$ be an object in a groupoid $\G$. The \emph{isotropy} group $\G_x$ at $x$ is the group of morphisms from $x$ to itself. A \emph{weak Garside group} is a group isomorphic to the isotropy group of an object in a Garside groupoid. A weak Garside group has \emph{finite type} if its associated Garside category has finite type.

A \emph{Garside monoid} is a Garside category with a single object and a \emph{Garside group} is a Garside groupoid with a single object.
\end{definition}

It follows from \cite[Proposition~3.11]{dehornoy2015foundations} that if $\G$ is a Garside groupoid, then the natural functor $\C\to \G$ is injective. Moreover, by \cite[lemma~3.13]{dehornoy2015foundations}, the triple $(\C,\C\stackrel{\phi}{\to}\C,1_\C \stackrel{\Delta}{\Rightarrow}\phi)$ naturally extends to $(\G,\G\stackrel{\phi}{\to}\G,1_\G \stackrel{\Delta}{\Rightarrow}\phi)$ (we denote the extensions of $\phi$ and $\Delta$ by $\phi$ and $\Delta$ as well).

Define $\G_{x\to}$ and $\G_{\to x}$ in a similar way to $\C_{x\to}$ and $\C_{\to x}$. For two morphisms $x$ and $y$ of $\G$, define $x\preccurlyeq y$ if there exists a morphism $z$ of $\C$ such that $y=xz$. Define $y\succcurlyeq x$ analogously. The following Lemma~\ref{lem:lattice hom} is well-known for inclusion from a Garside monoid to its associated Garside group. The more general case concerning Garside category can be proved similarly.

\begin{lemma}
	\label{lem:lattice hom}
For each object $x$ of $\G$, $\G_{x\to}$ and $\G_{\to x}$ are lattices. Moreover, the natural inclusions $\C_{x\to}\to \G_{x\to}$ and $\C_{\to x}\to \G_{\to x}$ are homomorphisms between lattices.
\end{lemma}

For morphisms $f_1,f_2\in \G_{x\to}$ (resp.\ $\G_{\to x}$), we use $f_1\wedge_p f_2$ (resp.\ $f_1\wedge_s f_2$) to denote the greatest lower bound of $f_1$ and $f_2$ with respect to $\preccurlyeq$ (resp.\ $\succcurlyeq$), where $p$ (resp.\ $s$) stands for prefix (resp.\ suffix) order. Similarly, we define $f_1\vee_p f_2$ and $f_1\vee_s f_2$.

\begin{lemma}
	\label{lem:basic properties}
Let $\G$ be a Garside groupoid and let $f:x\to y$ be a morphism of $\G$. Then
\begin{enumerate}
	\item the map $(\G_{y\to},\preccurlyeq)\to (\G_{x\to},\preccurlyeq)$ mapping $h\in \G_y$ to $fh$ is a lattice monomorphism;
	\item the bijection $(\G_{x\to},\preccurlyeq)\to (\G_{\to x},\succcurlyeq)$ mapping $h\in \G_x$ to $h^{-1}$ is order-revering, i.e.\ $h_1\preccurlyeq h_2$ if and only if $h^{-1}_1\succcurlyeq h^{-1}_2$;
	\item choose elements $\{a_i\}_{i=1}^n\subset \G_{\to x}$, then $a^{-1}_1\vee_p a^{-1}_2\vee_p\cdots\vee_p a^{-1}_n=(a_1\wedge_s a_2\wedge_s\cdots\wedge_s a_n)^{-1}$ and $a^{-1}_1\wedge_p a^{-1}_2\wedge_p\cdots\wedge_p a^{-1}_n=(a_1\vee_s a_2\vee_s\cdots\vee_s a_n)^{-1}$.
\end{enumerate}
\end{lemma}

\begin{proof}
Assertion (1) follows from the definition. To see (2), note that if $h_1\preccurlyeq h_2$, then there exists a morphism $k$ of $\C$ such that $h_2=h_1 k$. Thus $h^{-1}_1=k h^{-1}_2$. Hence $h^{-1}_1\succcurlyeq h^{-1}_2$. The other direction is similar. (3) follows from (2).
\end{proof}

\begin{lemma}
	\label{lem:order}
	Let $\G$ be a Garside groupoid. Let $\{a_i\}_{i=1}^n$ and $c$ be morphisms in $\G$. Let $\{b_i\}_{i=1}^n$ be morphisms in $\C$.
	\begin{enumerate}
		\item Suppose $a_ib_i=c$ for $1\le i\le n$. Let $a=a_1\vee_p a_2\vee_p\cdots\vee_p a_n$ and $b=b_1\wedge_s b_2\wedge_s\cdots\wedge_s b_n$. Then $ab=c$.
		\item Suppose $a_ib_i=c$ for $1\le i\le n$. Let $a=a_1\wedge_p a_2\wedge_p\cdots\wedge_p a_n$ and $b=b_1\vee_s b_2\vee_s\cdots\vee_s b_n$. Then $ab=c$.
		\item $(cb_1)\wedge_p(cb_2)\wedge_p\cdots\wedge_p(cb_n)=c(b_1\wedge_p\cdots \wedge_p b_n)$.
		\item $(cb_1)\vee_p(cb_2)\vee_p\cdots\vee_p(cb_n)=c(b_1\vee_p\cdots \vee_p b_n)$.
	\end{enumerate}
\end{lemma}

\begin{proof}
For (1), note that all the $a_i$ have the same source object and all the $b_i$ have the same target object. Thus $a$ and $b$ are well-defined. It follows from Lemma~\ref{lem:basic properties} (1) that $c^{-1}a=(c^{-1}a_1)\vee_p (c^{-1}a_2)\vee_p\cdots\vee_p (c^{-1}a_n)=b^{-1}_1\vee_p b^{-1}_2\vee_p\cdots\vee_p b^{-1}_n$. By Lemma~\ref{lem:basic properties}~(3), $b^{-1}_1\vee_p\cdots\vee_p b^{-1}_n=(b_1\wedge_s b_2\wedge_s\cdots\wedge_s b_n)^{-1}=b^{-1}$. Thus (1) follows. Assertion (2) can be proved similarly. (3) and (4) follow from Lemma~\ref{lem:basic properties}~(1).
\end{proof}

For a morphism $f$ of $\G$, denote the source object and the target object of $f$ by $s(f)$ and $t(f)$.

\begin{lemma}
	\label{lem:simple element1}
Suppose $f$ is a morphism of $\G$. Then $f\preccurlyeq \Delta_{s(f)}$ if and only if $\Delta^{t(f)}\succcurlyeq f$. 
\end{lemma}

\begin{proof}
If $f\preccurlyeq \Delta_{s(f)}$, then $\Delta_{s(f)}=fg$ for $g$ in $\C$. Let $h$ be the unique morphism such that $g=\phi(h)$. By the definition of $\phi$, we know $h$ is in $\C$. Moreover, $h$ fits into the following commuting diagram:
\begin{center}
$\begin{CD}
	\phi^{-1}(t(f))                        @>h>>         s(f)\\
	@VV\Delta^{t(f)}V                         @VV\Delta_{s(f)}V\\
	t(f)=s(g)                     @>g>>      t(g)=t(\Delta_{s(f)})
\end{CD}$
\end{center}
Thus $h\Delta_{s(f)}=\Delta^{t(f)}g$, which implies that $h(fg)=\Delta^{t(f)}g$. Hence $hf=\Delta^{t(f)}$ and $\Delta^{t(f)}\succcurlyeq f$. The other direction is similar.
\end{proof}

\begin{lemma}
	\label{lem:simple element2}
Let $x$ be an object of $\G$. Let $\Phi_1$ be the collection of simple morphisms of $\G$ with target object $x$ and let $\Phi_2$ be the collection of simple morphisms with source object $x$. For $a\in \Phi_1$, we define $a^\ast$ to be the morphism such that $aa^\ast=\Delta_{s(a)}$. Then 
	\begin{enumerate}
		\item the map $a\to a^\ast$ gives a bijection $\Phi_1\to \Phi_2$;
		\item for $a,b\in \Phi_1$, we have $b\succcurlyeq a$ if and only if $b^\ast\preccurlyeq a^\ast$;
		\item for $a,b\in \Phi_1$, we have $(a\wedge_s b)^\ast=a^\ast\vee_p b^\ast$ and $(a \vee_s b)^\ast=a^\ast \wedge_p b^\ast$.
	\end{enumerate}
\end{lemma}

\begin{proof}
For (1), since $\Delta^{t(a^\ast)}=\Delta_{s(a)}\succcurlyeq a^\ast$, we have $a^\ast\preccurlyeq \Delta_{s(a^\ast)}$ by Lemma~\ref{lem:simple element1}. Thus $a^\ast\in\Phi_2$. The map is clearly injective and surjectivity follows from Lemma~\ref{lem:simple element1}. For (2), $b\succcurlyeq a$ implies $b=ka$ for a morphism $k$ of $\C$. From $\Delta_{s(b)}=bb^\ast=kab^\ast$, we know $\Delta^{t(b^\ast)}=\Delta_{s(b)}\succcurlyeq ab^\ast$, hence $ab^\ast\preccurlyeq \Delta_{s(a)}$ by Lemma~\ref{lem:simple element1}. Then there is a morphism $h$ of $\C$ such that $ab^\ast h=\Delta_{s(a)}=aa^\ast$, hence $b^\ast h=a^\ast$ and $b^\ast\preccurlyeq a^\ast$. The other direction is similar. For (3), note that elements in $\Phi_1$ (resp.\ $\Phi_2$) are closed under $\wedge_s$ and $\vee_s$ (resp.\ $\wedge_p$ and $\vee_p$), so $\phi_1$ and $\Phi_2$ are lattices. Now assertion (3) follows from (2) and (1).
\end{proof}
%Indeed, since $a\succcurlyeq (a\wedge_s b)$, we have $a^\ast\preccurlyeq (a\wedge_s b)^\ast$ by (2). Similarly, $b^\ast\preccurlyeq (a\wedge_s b)^\ast$, hence $a^\ast\vee_p b^\ast\preccurlyeq (a\wedge_s b)^\ast$. On the other hand, $a^\ast\preccurlyeq a^\ast\vee_p b^\ast$, by (2) we have $a\succcurlyeq c$ where $c^\ast=a^\ast\vee_p b^\ast$. Similarly, $b\succcurlyeq c$. Hence $a\wedge_s b\succcurlyeq c$. By (2) again, $(a\wedge_s b)^\ast\preccurlyeq c^\ast=a^\ast\vee_p b^\ast$. The other equality of (3) is similar.

\subsection{Garside groups are Helly} 
\label{ss:Garside groups are Helly}
Let $\G$ be a Garside groupoid and let $x$ be an object of $\G$. A \emph{cell} of $\G_{x\to}$ is an interval of the form $[f,f\Delta_{t(f)}]$ with respect to the order $\preccurlyeq$.
\begin{lemma}
	\label{lem:spherical triple intersection}
	Let $C_1,C_2$ and $C_3$ be pairwise intersecting cells of $\G_{x\to}$. Then there is a cell $C$ of $\G_{x\to}$ such that $D=(C_1\cap C_2)\cup(C_2\cap C_3)\cup (C_3\cap C_1)$ is contained in $C$.
\end{lemma}

\begin{proof}
	Since each $C_i$ is an interval in $(\G_{x\to},\preccurlyeq)$, we write $C_i=[f_i,g_i]$. Let $C_{ij}=C_i\cap C_j$. Then each $C_{ij}$ is also an interval (Lemma~\ref{lem:lattice} (1)). We write $C_{ij}=[f_{ij},g_{ij}]$. Then $f_{ij}=f_i\vee_p f_j$ and $g_{ij}=g_i\wedge_p g_j$.
	Choose a morphism $h$ in $\cap_{i=1}^3 C_i$ (Lemma~\ref{lem:lattice} (2)). Since $f_i\preccurlyeq h\preccurlyeq g_i$, there are morphisms $w_i$ and $w^\ast_i$ in $\C$ such that $h=f_iw_i$ and $g_i=hw^\ast_i$. Note that $w_iw^{\ast}_i=\Delta_{s(f_i)}$. Moreover, by Lemma~\ref{lem:order} (1) and (3), we have
	\begin{equation}
	\label{eq:1}
	f_{ij}\cdot (w_i \wedge_s w_j)=h\ \mathrm{and}\ h\cdot(w^{\ast}_i \wedge_p w^{\ast}_j)=g_{ij}.
	\end{equation}

	Let $f$ (resp.\ $g$) be the greatest lower bound (resp.\ least upper bound) over all elements in $D$ with respect to the prefix order $\preccurlyeq$. Since $D$ is a union of three intervals, we have $f=f_{12}\wedge_p f_{23}\wedge_p f_{31}$ and $g=g_{12}\vee_p g_{23}\vee_p g_{31}$. Define $a= (w_1 \wedge_s w_2) \vee_s (w_2 \wedge_s w_3) \vee_s (w_3 \wedge_s w_1)$
	and $b = (w^{\ast}_1 \wedge_p w^{\ast}_2) \vee_p (w^{\ast}_2 \wedge_p w^{\ast}_3) \vee_p (w^{\ast}_3 \wedge_p w^{\ast}_1)$. Then by the formula (\ref{eq:1}) above and Lemma~\ref{lem:order} (2) and (4), we have $fa=h$ and $hb=g$.
	
	We claim $ab\preccurlyeq \Delta_{s(a)}$. This claim implies that the cell $[f,f\Delta_{s(a)}]$ contains the interval $[f,g]$, hence contains $D$. Thus the lemma follows from the claim.
	
	Now we prove the claim. Since $w_i\preccurlyeq \Delta_{s(w_i)}$ for $1\le i\le 3$, by Lemma~\ref{lem:simple element1}, $\Delta_{t(w_i)}\succcurlyeq w_i$. Then $\Delta_{t(w_i)}\succcurlyeq a$. By Lemma~\ref{lem:simple element1} again, $a\preccurlyeq \Delta_{s(a)}$. Let $a^\ast$ be a morphism in $\C$ such that $aa^\ast=\Delta_{s(a)}$. It suffices to show $b\preccurlyeq a^{\ast}$. By repeatedly applying Lemma~\ref{lem:simple element2} (3), we have $a^{\ast}= (w_1 \wedge_s w_2)^\ast \wedge_p (w_2 \wedge_s w_3)^\ast \wedge_p (w_3 \wedge_s w_1)^\ast=(w^{\ast}_1 \vee_p w^{\ast}_2) \wedge_p (w^{\ast}_2 \vee_p w^{\ast}_3) \wedge_p (w^{\ast}_3 \vee_p w^{\ast}_1)$. Note that $b$ is the least upper bound of three terms, however, each of these terms is $\preccurlyeq$ every term of $a^{\ast}$, so each of these terms is $\preccurlyeq a^{\ast}$, hence $b\preccurlyeq a^{\ast}$.
\end{proof}

\begin{theorem}
	\label{t:Garside Helly}
	Let $G$ be a weak Garside group of finite type associated with a Garside groupoid $\G$. Then $G$ acts geometrically on a Helly graph.
\end{theorem}

\begin{proof}
Let $x$ be an object of $\G$. Suppose $G=\G_x$. Let $Y$ be a simplicial graph whose vertex set is $\G_{x\to}$ and two vertices are adjacent if they are in the same cell of $\G_{x\to}$. The left action $\G_x\curvearrowright\G_{x\to}$ preserves the collection of cells, hence induces an action $\G_x\curvearrowright Y$.

Note that $Y$ is uniformly locally finite. Indeed, given $f\in \G_{x\to}$, $g$ is adjacent to $f$ if and only if $g=fa^{-1}b$ for some simple morphisms $a$ and $b$ of $\C$. Since the collection of simple morphisms is finite, there are finitely many possibilities for such $g$. Thus $Y$ satisfies assumption (1) of Lemma~\ref{lem:helly criterion1} with $\{X_i\}_{i\in I}$ being cells of $\G_{x\to}$. Assumption (2) of Lemma~\ref{lem:helly criterion1} follows from Lemma~\ref{lem:lattice} (2), as each cell is an interval in the lattice $\G_{x\to}$. Assumption (3) of Lemma~\ref{lem:helly criterion1} follows from Lemma~\ref{lem:spherical triple intersection}. Thus $Y$ is finitely clique Helly by Lemma~\ref{lem:helly criterion1}. By Theorem~\ref{thm:Helly}, it remains to show that the flag complex $F(Y)$ of $Y$ is simply connected.
	
Let $Z$ be the cover graph of the poset $(\G_{x\to},\preccurlyeq)$, i.e.\ two vertices $f$ and $h$ are adjacent if $f=hk$ or $h=fk$ for some atom $k$. Since all atoms are simple, $Z$ is a subgraph of $Y$. As every morphism in $\C$ can be decomposed into a product of finitely many atoms by Definition~\ref{def:Garside} (2), each edge of $Y$ is homotopic rel its ends points in $F(Y)$ to an edge path in $Z$. Let $\omega\subset Y$ be an edge loop. We may assume $\omega\subset Z$ up to homotopy in $F(Y)$. Let $f\in\G_{x\to}$ be a common lower bound for the collection of vertices of $\omega$. We denote consecutive vertices of $\omega$ by $\{ff_i\}_{i=1}^n$ where each $f_i$ is a morphism of $\C$. We define the complexity $\tau(\omega)$ of $\omega$ to be $\max_{1\le i\le n}\{l(f_i)\}$, where $l$ is the length function from Section~\ref{subsec:Garside basic}. Suppose $l(f_{i_0})=\tau(\omega)$. Then, by definition of $l$, we have $l(f_{i_0-1})<l(f_{i_0})$ and $l(f_{i_0+1})<l(f_{i_0})$. Thus $f_{i_0}=f_{i_0-1}a_1$ and $f_{i_0}=f_{i_0+1}a_2$ for atoms $a_1$ and $a_2$ of $\C$. Let $f'_{i_0}=f_{i_0-1}\wedge_p f_{i_0+1}$. Let $\omega_1$ (resp.\ $\omega_2$) be the edge path in $Z$ from $f'_{i_0}$ to $f_{i_0-1}$ (resp.\ $f_{i_0+1}$) corresponding to a decomposition of $(f'_{i_0})^{-1}f_{i_0-1}$ (resp.\ $(f'_{i_0})^{-1}f_{i_0+1}$) into atoms ($\omega_i$ might be a point). We replace the sub-path $\overline{f_{i_0-1}f_{i_0}}\cup \overline{f_{i_0+1}f_{i_0}}$ of $\omega$ by $\omega_1\cup \omega_2$ to obtain a new edge loop $\omega'\subset Z$. Clearly $\tau(\omega')<\tau(\omega)$. Moreover, since $f'_{i_0}(a_1\vee_s a_2)=f_{i_0}$ (Lemma~\ref{lem:order}) and $a_1\vee_s a_2$ is simple, by Lemma~\ref{lem:simple element1} $\overline{f_{i_0-1}f_{i_0}}\cup \overline{f_{i_0+1}f_{i_0}}\cup \omega_1\cup \omega_2$ is contained in a simplex of $F(Y)$. Thus $\omega$ and $\omega'$ are homotopic in $F(Y)$. By repeating this process, we can homotop $\omega$ to a constant path, which finishes the proof.
\end{proof}

\begin{remark}
The simple connectedness of $F(Y)$ above be can alternatively deduced from \cite[Corollary~7.6]{bessis2006garside}, where Bessis considers a graph $\Gamma$ on $\G_{x\to}$ such that two vertices are adjacent if they differ by a simple morphism in $\C$. It is proved that the flag complex $F(\Gamma)$ is contractible. As $\Gamma$ is a proper subgraph of $Y$, it is not hard to deduce the simply-connectedness of $F(Y)$. The definition of $\Gamma$ goes back to work of Bestvina \cite{Bestvina1999} in the case of spherical Artin groups, and Charney-Meier-Whittlesey \cite{charney2004bestvina} in the case of Garside group. 
\end{remark}

For a braid group $G$, Theorem~\ref{t:Garside Helly} gives rise to many different Helly graphs (hence injective metric spaces) upon which $G$ acts geometrically. One can choose different Garside structures on $G$: either we use the interval $[1,\Delta^k]$ as a cell, or we use the dual Garside structure \cite{birman1998new,bessisdual}. One can also apply some combinatorial operations to a Helly graph --- e.g.\ taking the Rips complex or the face complex --- in order to obtain a new Helly graph (see \cite{Helly}).

\begin{cor}
	\label{cor:zonotopes}
Let $\mathcal{A}$ be a finite simplicial central arrangement of hyperplanes in a finite dimensional real vector space $V$. Let $\Sal(\mathcal{A})$ be its Salvetti complex and let $\widetilde{\Sal}(\mathcal{A})$ be the universal cover of $\Sal(\mathcal{A})$. Then $\widetilde{\Sal}(\mathcal{A})$ is cell Helly.
\end{cor}

\begin{proof}
We refer to \cite{Salvetti,ParisSalvetti} for relevant definitions and background. Let $M(\mathcal{A})=V_{\mathbb C}-\left(\bigcup_{H\in\mathcal{A}} H_{\mathbb C}\right) $ and $G=\pi_1 (M(\mathcal{A}))$. Then $G$ is a weak Garside group \cite{Deligne}. $G$ can be identified as vertices of $\widetilde{\Sal}(\mathcal{A})$, and a cell of $G$ in the sense of Lemma~\ref{lem:spherical triple intersection} can be identified with the vertex set of a top dimensional cell in $\widetilde{\Sal}(\mathcal{A})$, moreover, the vertex set of each cell of $\widetilde{\Sal}(\mathcal{A})$ can be identified with an interval in $G$. Thus Definition~\ref{d:cellHelly} (2) holds. Similarly to Lemma~\ref{lem:full}, one can prove that any closed cell in $\widetilde{\Sal}(\mathcal{A})$ is a full subcomplex. Hence any intersection of cells is a full subcomplex spanned by vertices in an interval of one cell. Hence Definition~\ref{d:cellHelly} (3) holds. Moreover, the fact that cells are full together with Lemma~\ref{lem:spherical triple intersection} imply that the collection of top dimensional cells in $\widetilde{\Sal}(\mathcal{A})$ satisfies Definition~\ref{d:cellHelly} (3). Therefore, the same holds for the collection of cells in $\widetilde{\Sal}(\mathcal{A})$ as every cell is contained in a top dimensional cell.
\end{proof}

%Note that $Y'$ is a subgraph of $Y$, hence we have a natural map their flag complexes $F(Y')\to F(Y)$. Let $\omega\subset Y$ be an edge path. Note that each edge of $Y$ is either an edge in $Y'$, or is contained in a triangle of $Y$ whose two other sides are edges in $Y'$. Thus we can homotopy $\omega$ inside $Y$ to another edge path $\omega'\subset Y$ such that $\omega'$ actually comes from $Y'$. By \cite[Thereom 3.1]{charney2004bestvina}, $F(Y')$ is contractible. Hence we can homotopy $\omega'$ to a point inside $F(Y')$, which also gives a homotopy in $F(Y)$ via the continuous map $F(Y')\to F(Y)$.

\section{Artin groups and Helly graphs}
\label{s:ArtinH}
\subsection{Helly property for cells in $X_\Gamma$}
Recall that $X^{(0)}_\Gamma$ is endowed with the prefix order $\preccurlyeq$. For a cell $C\in X_\Gamma$, we use $\preccurlyeq_C$ to denote the right weak order on $C^{(0)}$ (we view $C$ as an oriented Coxeter cell). And for a standard subcomplex $X_1\subset X_\Gamma$ of type $\Gamma_1$, we use $\preccurlyeq_{\Gamma_1}$ to denote the prefix order on $X^{(0)}_1$ coming from $A_{\Gamma_1}$.

\begin{lemma}
	\label{lem:interval}
	Let $C_1,C_2\subset X_\Gamma$ be two cells with non-empty intersection. Then the vertex set of $C_1\cap C_2$ is an interval in $(X^{(0)}_\Gamma,\preccurlyeq)$. 
\end{lemma}

\begin{proof}
	For $i=1,2$, suppose $C_i$ is of type $\Gamma_i$ and let $X_i$ be the standard subcomplex of type $\Gamma_i$ containing $C_i$. Let $X_0=X_1\cap X_2$, which is a standard subcomplex of type $\Gamma_0=\Gamma_1\cap \Gamma_2$. Let $D_i=C_i\cap X_0$. Then $D_i$ is a cell by Lemma~\ref{lem:intersection of cell and subcomplex}. Moreover, $D^{(0)}_1$ and $D^{(0)}_2$ are intervals in $(X^{(0)}_0,\preccurlyeq_{\Gamma_0})$ by Lemma~\ref{lem:isometric}. As $(X^{(0)}_0,\preccurlyeq_{\Gamma_0})$ is a lattice (Theorem~\ref{thm:lattice}), $(C_1\cap C_2)^{(0)}=(D_1\cap D_2)^{(0)}$ is an interval in $(X^{(0)}_0,\preccurlyeq_{\Gamma_0})$ (Lemma~\ref{lem:lattice}), hence it is also an interval in $(X^{(0)}_\Gamma,\preccurlyeq)$ by Lemma~\ref{lem:intersection of cell and subcomplex} (3) and (4). 
\end{proof}

\begin{lemma}
	\label{lem:contractible}
	Let $C$ be an oriented Coxeter cell. Let $K$ be the full subcomplex of $C$ spanned by vertices in an interval of $C^{(0)}$. Then $K$ is contractible.
\end{lemma}

\begin{proof}
Let $[x,y]$ be an interval of $C^{(0)}$. By \cite[Proposition 3.1.6]{bjorner2006combinatorics}, we can assume $x$ is the source of $C$. Let $\ell:C^{(0)}\to \mathbb Z$ be the function measuring the combinatorial distance from $x$ to a given point in $C^{(0)}$. We extend $\ell$ to $C\to \mathbb R$ such that $\ell$ is affine when restricted to each cell. Note that $\ell$ is non-constant on each cell with dimension $>0$. Then $\ell|_K$ is a Morse function in the sense of \cite[Definition 2.2]{MR1465330}. Moreover, it follows from \cite[Chapter IV, Exercise 22]{bourbaki2002lie} that the descending link (\cite[Definition 2.4]{MR1465330}) of each vertex for $\ell|_K$ is a simplex, hence contractible. Thus $K$ is contractible by \cite[Lemma 2.5]{MR1465330}.
\end{proof}

\begin{lemma}
	\label{lem:triple intersection1}
	Suppose $\Gamma$ is spherical. Let $C_1$ and $C_2$ be two cells in $X_\Gamma$, and let $X_0$ be a standard subcomplex of $X_\Gamma$ of type $\Gamma_0$. Suppose that $C_1,C_2$, and $X_0$ pairwise intersect. Then their triple intersection is nonempty.
\end{lemma}

\begin{proof}
	For $i=1,2$, let $v_i$ be a vertex of $C_i\cap X_0$. Since $(X^{(0)},\preccurlyeq_{\Gamma_0})$ is a lattice by Theorem~\ref{thm:lattice}, we can find $a,b\in X^{(0)}$ such that $a\preccurlyeq_{\Gamma_0} v_i\preccurlyeq_{\Gamma_0} b$ for $i=1,2$. Let $I$ be the interval between $a$ and $b$ with respect to the prefix order on $X^{(0)}_\Gamma$. Then $I\subset X_0$ by Lemma~\ref{lem:intersection of cell and subcomplex} (3). Note that $C_1,C_2$ and $I$ pairwise intersect by construction. Since $C^{(0)}_i$ is an interval in $(X^{(0)}_\Gamma,\preccurlyeq)$ for $i=1,2$ (Lemma~\ref{lem:isometric} (2)), and $(X^{(0)}_\Gamma,\preccurlyeq)$ is a lattice (recall that $\Gamma$ is spherical), we have $C_1\cap C_2\cap I\neq\emptyset$ by Lemma~\ref{lem:lattice} (2). Since $I\subset X_0$, the lemma follows.
\end{proof}

\begin{lemma}
	\label{lem:triple intersection2}
	Let $C_1,C_2,C_3$ be three cells in $X_\Gamma$. Suppose they pairwise intersect. Then $C_1\cap C_2\cap C_3\neq\emptyset$.
\end{lemma}

\begin{proof}
	Let $\Gamma_i$ be the type of $C_i$ and let $X_i$ be the standard subcomplex of type $\Gamma_i$ containing $C_i$. Let $\Gamma_0=\cap_{i=1}^3\Gamma_i$.
	
	\emph{Case 1: $\Gamma_1=\Gamma_2=\Gamma_3$.} Then $X_1=X_2=X_3$. Since each $C^{(0)}_i$ is an interval in $(X^{(0)}_1,\preccurlyeq_{\Gamma_1})$ (Lemma~\ref{lem:isometric}) and $(X^{(0)}_1,\preccurlyeq_{\Gamma_1})$ is a lattice (Theorem~\ref{thm:lattice}), the lemma follows from Lemma~\ref{lem:lattice}.
	
	\emph{Case 2: only two of the $\Gamma_i$ are equal}. We suppose without loss of generality that $\Gamma_1=\Gamma_2\neq\Gamma_3$. Let $C'_3=C_3\cap X_1$. Then $C'_3$ is a cell by Lemma~\ref{lem:intersection of cell and subcomplex}. Moreover, $C_1$, $C_2$ and $C'_3$ pairwise intersect. As the previous case, we know $C_1\cap C_2\cap C'_3\neq\emptyset$, hence the lemma follows.
	
	\emph{Case 3: the $\Gamma_i$ are pairwise distinct.} For $1\le i\neq j\le 3$, let $C_{ij}=C_i\cap X_j$ (see Figure~\ref{f:triple}). We claim $C_{12}\cap C_{13}$ is a cell of type $\Gamma_0$ (when $\Gamma_0=\emptyset$, the intersection is a vertex). Consider the cellular map $\pi:X_\Gamma\to D_\Gamma$ defined just before Lemma~\ref{lem:isometric}. Let $\bar C_i$ be the image of $C_i$. Note that $C_{12}$ is a cell of type $\Gamma_1\cap\Gamma_2$ (Lemma~\ref{lem:intersection of cell and subcomplex}), and $\bar C_1\cap \bar C_2$ is a cell of type $\Gamma_1\cap\Gamma_2$ (Theorem~\ref{thm:lek}), hence $\pi(C_{12})=\bar C_1\cap \bar C_2$. Similarly, $\pi(C_{13})=\bar C_1\cap \bar C_3$. Since $\{\bar C_i\}_{i=1}^3$ pairwise intersect, their triple intersection is nonempty by Lemma~\ref{lem:Coxeter helly}. Thus $\pi(C_{12})\cap \pi(C_{13})\neq\emptyset$. Hence $\pi(C_{12})\cap \pi(C_{13})$ is a cell of type $\Gamma_0$. Now, the claim follows from the fact that $\pi$ restricted to $C_1$ is an embedding.
	
	\begin{figure}[h!]
		\centering
		\includegraphics[width=0.38\textwidth]{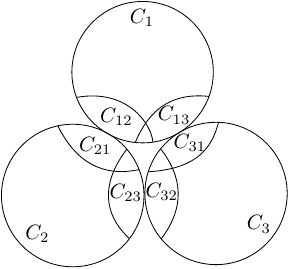}
		\caption{}
		\label{f:triple}
	\end{figure}
	
	For $i=1,2,3$, let $\tau_i=C_{i,i+1}\cap C_{i,i+2}$ (indices are taken mod 3). By the previous paragraph, each $\tau_i$ is a cell of type $\Gamma_0$. Let $X_0=\cap_{i=1}^3 X_i$. Then $\tau_i\subset X_0$ for $1\le i\le 3$. In particular $X_0$ is nonempty, hence is a standard subcomplex of type $\Gamma_0$ by Lemma~\ref{lem:intersection of cell and subcomplex} (1). Now we claim $\cap_{i=1}^3\tau_i\neq \emptyset$. Note that the lemma follows from this claim since $\tau_i\subset C_i$ for $1\le i\le 3$.
	
	Since each $\tau^{(0)}_i$ is an interval in the lattice $(X^{(0)}_0,\preccurlyeq_{\Gamma_0})$, it suffices to show that $\tau_i$ pairwise intersect. We only prove $\tau_1\cap\tau_2\neq\emptyset$, as the other pairs are similar. Now, consider the triple $\{C_{12},C_{21},X_0\}$. Note that $\tau_1\subset C_{12}$, $\tau_2\subset C_{21}$ and $\tau_i\subset X_0$ for $i=1,2$. Therefore $\tau_1\cap\tau_2\subset C_{12}\cap C_{21}\cap X_0$. On the other hand, since $C_{12}$ is a cell of type $\Gamma_1\cap\Gamma_2$ and $C_{12}\cap X_0\supset \tau_1\neq\emptyset$, we have that $C_{12}\cap X_0$ is a cell of type $\Gamma_0$ from Lemma~\ref{lem:intersection of cell and subcomplex}, hence $C_{12}\cap X_0=\tau_1$. Similarly $C_{21}\cap X_0=\tau_2$. Thus actually $\tau_1\cap\tau_2= C_{12}\cap C_{21}\cap X_0$. It remains to show that $C_{12}\cap C_{21}\cap X_0\neq\emptyset$ This follows from Lemma~\ref{lem:triple intersection1}, since $C_{12},C_{21},X_0$ are contained in $X_1\cap X_2$, which is a spherical standard subcomplex.
\end{proof}

\begin{prop}
	\label{prop:helly general}
	Let $\{C_i\}_{i=1}^n$ be a finite collection of cells in $X_\Gamma$. Suppose they pairwise intersect. Then $\cap_{i=1}^n C_i$ is a nonempty contractible subcomplex of $X_\Gamma$.
\end{prop}

\begin{proof}
	For $2\le i\le n$, let $\tau_i=C_1\cap C_i$. Then $\tau^{(0)}_i$ is an interval in $(X^{(0)}_\Gamma,\preccurlyeq)$, hence it is also an interval in $(C^{(0)}_1,\preccurlyeq_C)$ by Lemma~\ref{lem:isometric} (2). Moreover, it follows from Lemma~\ref{lem:triple intersection2} that $\{\tau_i\}_{i=2}^n$ pairwise intersect. Since $(C^{(0)}_1,\preccurlyeq_C)$ is a lattice (Theorem~\ref{thm:lattice} (1)), we have that $\cap_{i=2}^n\tau^{(0)}_i$ is a nonempty interval in $(C^{(0)}_1,\preccurlyeq_C)$ by Lemma~\ref{lem:lattice}, hence $\cap_{i=1}^n C_i\neq\emptyset$. Since each $C_i$ is a full subcomplex of $X_\Gamma$ (Lemma~\ref{lem:full}), $\cap_{i=1}^n C_i$ is full in $X_\Gamma$ (hence is full in $C_1$). Now, contractibility of $\cap_{i=1}^n C_i$ follows from Lemma~\ref{lem:contractible}.
\end{proof}

\subsection{Artin groups of type FC act on Helly graphs}
\label{ss:FC}
\begin{lemma}
	\label{lem:spherical triple intersection1}
	Suppose $\Gamma$ is spherical. Let $C_1,C_2$ and $C_3$ be pairwise intersecting cells in $X_\Gamma$. Then there is a cell $C\subset X_\Gamma$ such that $(C_1\cap C_2)\cup(C_2\cap C_3)\cup (C_3\cap C_1)\subseteq C$.
\end{lemma}

\begin{proof}
It suffices to consider the special case when each $C_i$ is a maximal cell of $X_\Gamma$. Consider the cell $C$ of type $\Gamma$ in $X_\Gamma$ with the source vertex being the identity. Let $\Delta$ be the sink vertex of $C$. Then by \cite{brieskorn1972artin,Deligne}, $A_\Gamma$ is a Garside group with $A^+_\Gamma$ being its Garside monoid and $\Delta$ being its Garside element (the $\phi$ in Definition~\ref{def:Garside} corresponds to conjugating by $\Delta$). Moreover, cells in the sense of Lemma~\ref{lem:spherical triple intersection} correspond to vertices of maximal cells of $X_\Gamma$. Thus the lemma follows from Lemma~\ref{lem:spherical triple intersection} and Lemma~\ref{lem:full}.
\end{proof}

\begin{prop}
	\label{prop:triple FC}
	Suppose that $\Gamma$ is of type FC. Let $C_1,C_2$ and $C_3$ be maximal pairwise intersecting cells in $X_\Gamma$. Then there is a maximal cell $C\subset X_\Gamma$ such that $D=(C_1\cap C_2)\cup(C_2\cap C_3)\cup (C_3\cap C_1)\subseteq C$.
\end{prop}

\begin{proof}
	For $1\le i\le 3$, suppose $C_i$ is of type $\Gamma_i$ and let $X_i$ be the standard spherical subcomplex of type $\Gamma_i$. If $X_1=X_2=X_3$, then the proposition follows from Lemma~\ref{lem:spherical triple intersection1}.
	
	Suppose two of the $X_i$ are the same. We assume without loss of generality that $X_1=X_2\neq X_3$. Let $C'_3=C_3\cap X_1$. Then $C_i\cap C_3=C_i\cap C'_3$ for $i=1,2$. Hence $D\subset (C_1\cap C_2)\cup(C_2\cap C'_3)\cup (C'_3\cap C_1)$. By Lemma~\ref{lem:intersection of cell and subcomplex} (2), $C'_3$ is a cell in $X_1$. Then we are done by Lemma~\ref{lem:spherical triple intersection1}.
	
	It remains to consider the case when $X_1,X_2$, and $X_3$ are mutually distinct. By Proposition~\ref{prop:helly general}, there is a vertex $p$ in $\cap_{i=1}^3 C_i$. Let $C_{ij}=C_i\cap C_j$ and $X_{ij}=X_i\cap X_j$. Then $X_{ij}$ is a spherical standard subcomplex of type $\Gamma_{ij}=\Gamma_i\cap \Gamma_j$ by Lemma~\ref{lem:intersection of cell and subcomplex} (1) (but $C_{ij}\subset X_{ij}$ might not be a cell in general). Since $\Gamma$ is of type FC, there is a spherical full subgraph $\Gamma_0\subset \Gamma$ such that $\Gamma_0$ contains $\Gamma_{12}\cup \Gamma_{23}\cup \Gamma_{31}$. Let $X_0$ be the spherical standard subcomplex of $X_\Gamma$ of type $\Gamma_0$ that contains $p$. Then for $1\le i\neq j\le 3$, we have $p\in C_i\cap C_j\subset X_{ij}\subset X_0$. It follows that $D\subset (C'_1\cap C'_2)\cup(C'_2\cap C'_3)\cup (C'_3\cap C'_1)$ where $C'_i=C_i\cap X_0$. Since each $C'_i$ is a cell (Lemma~\ref{lem:intersection of cell and subcomplex} (2)) and $X_0$ is spherical, the proposition follows from Lemma~\ref{lem:spherical triple intersection1}.
\end{proof}

\begin{theorem}
	\label{t:cHFC}
	Suppose $\Gamma$ is of type FC. 
	Then $(X_{\Gamma},\{C_i\})$ is a locally bounded cell Helly complex and, consequently, $A_\Gamma$ acts geometrically on a Helly graph
	being the $1$-skeleton of the thickening of $(X_{\Gamma},\{C_i\})$. Moreover, $X(\Gamma)$ is contractible.
\end{theorem}

\begin{proof}
	The conditions (1) and (2) of Definition~\ref{d:cellHelly} are satisfied by Proposition~\ref{prop:helly general}. The condition (3) holds by Proposition~\ref{prop:triple FC}. Since $X_\Gamma$ is simply connected,
	the theorem holds by Theorem~\ref{t:cH>H} and Proposition~\ref{p:Helly_action}. Contractibility follows from Proposition~\ref{prop:helly general} and Proposition~\ref{p:contract}.
\end{proof}

%\begin{theorem}
%	Let $Y_\Gamma$ be a simplicial graph whose vertex set is $X^{(0)}_\Gamma$, and two vertices are adjacent if they are contained in the same cell of $X_\Gamma$. Suppose $\Gamma$ is of type FC. Then the following holds:
%	\begin{enumerate}
%		\item $Y_\Gamma$ is a Helly graph, hence $A_\Gamma$ acts geometrically on a Helly graph;
%		\item the flag complex $Z_\Gamma$ of $Y_\Gamma$ is homotopy equivalent to $X_\Gamma$, hence $X_\Gamma$ is contractible.
%	\end{enumerate}
%\end{theorem}
%
%\begin{proof}
%	Since $X_\Gamma$ is uniformly locally finite, $Y_\Gamma$ is uniformly locally finite. Thus condition (1) of Lemma~\ref{lem:helly criterion1} is satisfied with $X=X^{(0)}_\Gamma$ and $\{X_i\}_{i\in I}$ being the vertex sets of cells of $X_\Gamma$. Conditions (2) and (3) of Lemma~\ref{lem:helly criterion1} follows from Proposition~\ref{prop:helly general} and Proposition~\ref{prop:triple FC} respectively. Thus $Y_\Gamma$ is finitely clique Helly by Lemma~\ref{lem:helly criterion1}. By Theorem~\ref{thm:Helly}, it remains to show that the flag complex $Z_\Gamma$ of $Y_\Gamma$ is simply connected.
%\end{proof}

%\begin{enumerate}
%	\item citation to Nima's paper;
%	\item Theorem 3.9 no infinite clique.
%\end{enumerate}
%=======================================================================

\bibliography{mybib}{}
\bibliographystyle{plain}

\end{document}